\documentclass[11pt,a4paper]{amsart}

\usepackage[utf8]{inputenc}
\usepackage[english]{babel}
\usepackage[			
	colorlinks=true,
	urlcolor=blue,
	linkcolor=green
]{hyperref}
\usepackage{csquotes}

\usepackage{amsmath}
\usepackage{amsfonts} 
\usepackage{amsthm}
\usepackage{amssymb}
\usepackage{mathtools}
\usepackage{mathrsfs}
\usepackage{mathdots} 
\usepackage{tikz}
\usepackage{tikz-cd}

\theoremstyle{plain}
\newtheorem{lemm}{Lemma}[section]
\newtheorem{theo}[lemm]{Theorem}
\newtheorem{prop}[lemm]{Proposition} 
\newtheorem{coro}[lemm]{Corollary}

\theoremstyle{definition}
\newtheorem{defi}[lemm]{Definition}

\theoremstyle{remark}
\newtheorem{rema}[lemm]{Remark}
\newtheorem{exam}[lemm]{Example}

\newcommand{\N}{\mathbb{N}}
\newcommand{\Z}{\mathbb{Z}}
\newcommand{\R}{\mathbb{R}}
\newcommand{\C}{\mathbb{C}}

\renewcommand{\d}{\mathrm{\ d}}

\newcommand{\delb}{\bar{\partial}}
\newcommand{\End}{\mathrm{End}}
\newcommand{\Fix}{\mathsf{Fix}}
\newcommand{\Eig}{\mathsf{Eig}}

\newcommand{\id}{\mathrm{id}}
\newcommand{\tr}{\mathrm{tr}}
\renewcommand{\ker}{\mathsf{ker}}
\newcommand{\Ker}{\mathsf{ker}}
\newcommand{\rest}{\: \rule[-3.5pt]{0.5pt}{11.5pt}\,{}}
\renewcommand{\H}{\mathsf{Hit}}
\newcommand{\supp}{\mathsf{supp}}
\renewcommand{\O}{\mathcal{O}}
\newcommand{\Heck}{\mathsf{Heck}}
\renewcommand{\Pr}{\mathcal{\mathcal{P}}}

\newcommand{\Prym}{\mathsf{Prym}}
\newcommand{\Pic}{\mathsf{Pic}}
\newcommand{\ord}{\mathsf{ord}}
\renewcommand{\div}{\mathsf{div}}
\newcommand{\Div}{\mathsf{Div}}
\newcommand{\Aut}{\mathsf{Aut}}
\newcommand{\A}{\mathcal{A}}
\newcommand{\M}{\mathcal{M}}
\newcommand{\Jac}{\mathsf{Jac}}
\newcommand{\tpst}{\tilde{\pi}^*}
\newcommand{\even}{\mathrm{even}}
\newcommand{\odd}{\mathrm{odd}}
\renewcommand{\P}{\mathbb{P}}
\renewcommand{\S}{\mathcal{S}}
\newcommand{\SL}{\mathsf{SL}}
\newcommand{\Nm}{\mathsf{Nm}}

\begin{document}
\title[Semi-Abelian Spectral Data]{Semi-Abelian Spectral Data for Singular Fibres of the \texorpdfstring{$\SL(2,\C)$}{SL(2,C)}-Hitchin System}


\author{Johannes Horn}

\address{%
Universität Heidelberg,
Im Neuenheimer Feld 205,
69120 Heidelberg, Germany
}
\email{johannes-horn@gmx.de}

\begin{abstract}
We describe spectral data for singular fibres of the $\SL(2,\C)$-Hitchin fibration with irreducible and reduced spectral curve. Using Hecke transformations we give a stratification of these singular spaces by fibre bundles over Prym varieties. By analysing the parameter spaces of Hecke transformations this describes the singular Hitchin fibres as compactifications of abelian group bundles over abelian torsors. We prove that a large class of singular fibres are themselves fibre bundles over Prym varieties. As applications we study irreducible components of singular Hitchin fibres and give a description of $\SL(2,\R)$-Higgs bundles in terms of these semi-abelian spectral data.
\end{abstract}

\maketitle

\section{Introduction}
Since the introduction of Higgs bundles in the late 80's by Hitchin \cite{Hi87a} and Simpson \cite{Si88}, the geometry of Higgs bundle moduli spaces is an extensively studied research topic at the interface of algebraic geometry, differential geometry and mathematical physics. From the algebro-geometric perspective the moduli space of Higgs bundles is a non-compact analogue of the moduli space of holomorphic vector bundles. In addition, it can be interpreted as the moduli space of solutions to a differential geometric PDE, the Hitchin equation. Moreover, it enables the study of moduli spaces of representations of the fundamental group of an orientable surface by the so-called Non-Abelian Hodge Correspondence. Thereby, Higgs bundles are an important tool to study components of these moduli spaces of representations in the realm of higher Teichmüller theory \cite{Hi92,BGG06,AlCo19,ABCGGO}. 

In the last ten years, there is renewed interest in Higgs bundles motivated by physics. Higgs bundle moduli spaces are hyperkähler and hence subject to mirror symmetry conjectures. Emerging from a duality between electric and magnetic charges in physics, mirror symmetry was interpreted as a duality of torus fibrations \cite{SYZ} and connected to the geometric Langlands program \cite{KaWi07}. 

The moduli space of $G$-Higgs bundles associated to a complex reductive Lie group $G$ contains a dense subset fibred by complex Lagrangian tori, the so-called Hitchin fibration. Mirror symmetry reincarnates as a duality between the Hitchin fibrations of the moduli space of $G$-Higgs bundles and the moduli space of $G^L$-Higgs bundles, where $G^L$ denotes the Langlands dual group of $G$. This Langlands duality was established for classical Lie groups by Hausel, Thaddeaus \cite{HaTh03} and Hitchin \cite{Hi07} and in general by Donagi, Pantev \cite{DoPa12}. 

In this paper, we study singular Hitchin fibres - degenerations of the Lagrangian tori forming the Hitchin fibration. Motivated by Langlands duality, singular Hitchin fibres were more closely studied in \cite{HaPa12, GO13, Br18, FGOP18, BBS19}. We give a complex-geometric description of these singular spaces emphasizing what is left of the rich geometry of the regular fibres. Indeed, singular Hitchin fibres naturally contain abelian subvarieties and one can describe their geometry building on this observation. These complex subtori played an important role in the recent work \cite{Hi19}, where Hitchin studied subintegrable systems on the singular locus of the Hitchin fibration. \\

Let $X$ be a Riemann surface of genus $g \geq 2$. Let $\M$ denote the moduli space of polystable $\SL(2,\C)$-Higgs bundles on $X$. In this case, the Hitchin map is given by 
\[ \H: \M \rightarrow H^0(X,K_X^2), \quad (E,\Phi) \mapsto \det(\Phi).
\] It was shown in \cite{Hi87a, Hi87b}, that this defines an algebraically completely integrable system over the subset of quadratic differentials with simple zeros. This defines the fibration by complex Lagrangian tori, we referred to as the Hitchin fibration. We will call the subset of quadratic differentials with simple zeros the regular locus and its complement the singular locus. 

To every quadratic differential $q_2 \in H^0(X,K_X^2)$, one associates a complex curve $\Sigma \subset \mathsf{Tot}(K_X)$, the so-called spectral curve. For a quadratic differential in the regular locus it is smooth and defines a branched covering of Riemann surfaces $\pi: \Sigma \rightarrow X$. In this case, $\H^{-1}(q_2)$ can be identified with an abelian torsor over the Prym variety associated to the spectral covering $\pi$. Hence, the complex tori forming the integrable system are abelian torsors.

The singular Hitchin fibres are divided in 3 types, depending on the regularity of the spectral curve. The most singular fibre is the nilpotent cone $\H^{-1}(0)$. It has many irreducible components, which intersect in a complicated way \cite{Hi87a}. A more detailed study can be found in the recent work \cite{ALS20}. Secondly, there are singular fibres, where the spectral curve decomposes into two irreducible components. I.e. there exists an abelian differential $\lambda \in H^0(X,K_X)$, such that $\lambda^2=q_2$. These fibres are still quite complicated.  
It was shown in \cite{GO13}, that they are connected. We study the case of $\SL(2,\C)$-Hitchin fibres with irreducible and reduced spectral curve. This is equivalent to $q_2$ having no global square root on $X$. By a classical theorem \cite{BNR89}, these singular Hitchin fibres can be identified with certain moduli spaces of rank 1 torsion-free sheaves on $\Sigma$. These moduli spaces were studied in \cite{GO13}, where again it is proven that they are connected. \\

We take a different approach. The normalisation naturally associates to the singular curve $\Sigma$ a smooth curve $\tilde{\Sigma}$. The Prym variety $\Prym(\tilde{\Sigma})$ associated to the covering $\tilde{\pi}: \tilde{\Sigma} \rightarrow X$ is naturally embedded in the singular Hitchin fibre through the pushforward along $\tilde{\pi}$. We describe the singular Hitchin fibre by gluing together holomorphic fibre bundles over the Prym varieties $\Prym(\tilde{\Sigma})$. 

\begin{theo}[Thm. \ref{theo:stratifi}, Thm. \ref{theo:strat:sl2C}] \label{theo:intro1}
Let $q_2 \in H^0(X,K_X^2)$ a quadratic differential with at least one zero of odd order. There is a stratification 
\[ \H^{-1}(q_2) = \bigcup\limits_{i \in I} \S_i 
\] by finitely many locally closed analytic subsets $\S_i$. Each stratum $\S_i$ is a holomorphic $(\C^*)^{r_i} \times \C^{s_i}$-bundle over $\Prym(\tilde{\Sigma})$ with $r_i,s_i \in \N_0$.
\end{theo}
We prove a more general version of this result in Theorem \ref{theo:strat:sl2C} covering all Hitchin fibres with irreducible and reduced spectral curve. The stratification is indexed by the local shape of the Higgs field at the higher order zeros of $q_2$. There is a unique open stratum $\S_0$, which is dense in the singular Hitchin fibre and is compactified by lower-dimensional strata. 

The $(\C^*)^{r_i} \times \C^{s_i}$-fibres - the non-abelian part of the spectral data - parametrize Hecke transformations. Originally introduced in \cite{NaRa75}, a Hecke transformation of a holomorphic vector bundle is the generalisation to higher rank of twisting a line bundle by a divisor. Hecke transformations are parametrized by Hecke parameters determining the direction, in which the holomorphic vector bundle is twisted. Hwang-Ramanan \cite{HwRa04} studied Hecke transformations of order one. Here the moduli space of Hecke parameters is one-dimensional and Hecke transformations define so-called Hecke curves in the moduli space of holomophic vector bundles. They showed that the space of tangents to Hecke curves is dual to the space of cotangent Higgs bundles lying over the singular locus. This result indicated a strong relationship between Hecke transformations and the geometry of singular Hitchin fibres. 

In this work, we use a generalisation of this concept by allowing higher order twists. This yields higher dimensional moduli spaces of Hecke parameters, which are the fibres of the strata $\S_i$ in Theorem \ref{theo:intro1} as bundles over Prym varieties. Hecke transformations of holomorphic vector bundles and Higgs bundles were used before to study solution spaces to geometric PDEs \cite{Ba10, HeWa19} and to parametrize moduli spaces of holomorphic $G$-bundles \cite{Wo13}. 

We study how the different strata are glued together by compactifying the space of Hecke parameters of the open stratum $\S_0$. These compact moduli spaces of Hecke parameters are given in terms of a quotient by a non-reductive group action. We can build a model for these quotient spaces by explicitly computing invariant functions (Theorem \ref{theo:moduli_Hecke}). For quadratic differentials with only zeros of odd order, the fibre bundle structure of $\S_0$ is preserved under the degeneration to lower-dimensional strata. We obtain the following theorem.

\begin{theo}[Thm. \ref{theo:anal:fibreing} ] \label{theo:intro2}
Let $q_2\in H^0(X,K_X^2)$ be a quadratic differential, such that all zeros have odd order. Then $\H^{-1}(q_2)$ fibres holomorphically over $\Prym(\tilde{\Sigma})$ with fibres given by the compact moduli of Hecke parameters. 
\end{theo}

For quadratic differentials with zeros of even order, we can still understand how the strata are glued together by analysing the compact moduli of Hecke parameters (see Section \ref{sec:global_even}). However, we observe that there can not be a global fibreing of $\H^{-1}(q_2)$ over $\Prym(\tilde{\Sigma})$ in this case. Nevertheless, we can study the irreducible components of the singular Hitchin fibres.

\begin{coro}[Cor. \ref{coro:even_irred}, Thm. \ref{theo:red:even:zero}] Let $q_2 \in H^0(X,K_X^2)$ with at least one zero of odd order, then $\H^{-1}(q_2)$ is an irreducible complex space. If all zeros of $q_2$ are of even order, then $\H^{-1}(q_2)$ is connected and has four irreducible components.
\end{coro}
These results will be generalized to a description of singular Hitchin fibres of the $\mathsf{Sp}(2n,\C)$-Hitchin system in a subsequent paper \cite{Ho2}.\\

To emphasize the analogies to the spectral data of regular Hitchin fibres, we generalize the correspondence between $\SL(2,\R)$-Higgs bundles and two-torsion points of $\Prym(\Sigma)$ (\cite{Sc13b,PNThesis,GPR18}) to singular Hitchin fibres.

\begin{theo}[Thm. \ref{theo:realpts:odd}] Let $q_2 \in H^0(X,K_X^2)$ be a quadratic differential, such that all zeros have odd order. For every stratum $\S_i \subset \H^{-1}(q_2)$, the $\SL(2,\R)$-Higgs bundles in $\S_i$ correspond to the two-torsion points of the Prym variety $\Prym(\tilde{\Sigma})$.
\end{theo}
Counting the two-torsion points of $\Prym(\tilde{\Sigma})$, this describes the branching of the moduli spaces of $\SL(2,\R)$-Higgs bundles as a covering over the Hitchin base (cf. \cite{Sc13a}). 
A similar result for quadratic differentials with even zeros is proven in Theorem \ref{theo:realpts:even}. \\

\noindent \textbf{Structure} \quad The paper is structured as follows: After introducing the framework (Section \ref{sec:prel}), we identify the Hitchin fibres with irreducible and reduced spectral curve with certain moduli spaces of Higgs bundles on $\tilde{\Sigma}$ (Section \ref{sec:pap}). In Section \ref{sec:Hecke}, we introduce the concept of Hecke transformations and adapt it to our setting. We prove the stratification result in Section \ref{sec:strat} and \ref{sec:stra:sl2c}. In Section \ref{sec:glob:odd}, we analyse the compact moduli of Hecke parameters for singular fibres over quadratic differentials, such that all zeros are of odd order, and prove Theorem \ref{theo:intro2}. Hereafter, we give an explicit description of the first degenerations (Section \ref{ssec:exam}). In Section \ref{sec:global_even}, we do the same analysis for general Hitchin fibres with irreducible and reduced spectral curve. In the final section, we parametrize the real points in singular Hitchin fibres by these semi-abelian spectral data (Section \ref{sec:real}). 


\subsection*{Funding}
This work was supported by the Deutsche Forschungsgemeinschaft (DFG, German Research Foundation) [281869850 (RTG 2229)]; the Klaus Tschira Foundation; and the U.S. National Science Foundation [DMS 1107452, 1107263, 1107367 "RNMS: GEometric structures And Representation varieties" (the GEAR Network)].

\subsection*{Acknowledgement} Great thanks to Daniele Alessandrini for his enduring support through the preparation of this work. Thank you to Richard Wentworth, Steven Bradlow, Xuesen Na and Andrew Sanders for enlightening discussions. We want to thank the referee for a careful reading of the manuscript and International Mathematics Research Notices for acceptance for publication.


\section{Preliminaries}\label{sec:prel}
Let $X$ be a Riemann surface of genus $g \geq 2$. Let $p_{K}: K= K_X \rightarrow X$ denote the line bundle of holomorphic 1-forms. Denote by $\mathcal{M}=\mathcal{M}(X)$ the moduli space of polystable $\SL(2,\C)$-Higgs bundles on $X$. These are pairs $(E,\Phi)$ of a holomorphic vector bundle $E$ of rank 2 with $\det(E)=\O_X$ and a holomorphic section $\Phi$ of $\End(E)\otimes K$, such that $\tr(\Phi)=0$. The smooth locus of $\mathcal{M}$, the moduli space of stable $\SL(2,\C)$-Higgs bundles, is a complex manifold of dimension $6g-6$. It carries a holomorphic symplectic structure extending the holomorphic symplectic structure of the cotangent bundle to the moduli space of stable holomorphic vector bundles. The Hitchin map is defined by 
\begin{align*}
\H: \mathcal{M} \rightarrow H^0(X, K^2), \quad (E,\Phi) \mapsto \det(\Phi)
\end{align*} with $\dim H^0(X,K^2)=3g-3$. This map is proper, surjective and defines an algebraically completely integrable system on the dense subset of quadratic differentials with simple zeros (\cite{Hi87a}, \cite{Hi87b}). 

The torus fibres can be identified algebraically as follows.  The Hitchin map computes the coefficients of the characteristic polynomial of $(E,\Phi) \in \H^{-1}(q_2)$ 
\[ \eta^2+q_2.
\]Here $\eta$ can be interpreted as the tautological section $\eta: K \rightarrow p_K^*K$. The pointwise eigenvalues of the Higgs field form the complex analytic curve
\[ \Sigma:=Z_{K}(\eta^2+p_K^*q_2) \subset \mathsf{Tot}K.
\]
This is the so-called spectral curve. The projection $p_K$ restricts to a two-sheeted branched covering $\pi: \Sigma \rightarrow X$ with branch points at the zeros of $q_2$. The spectral curve is smooth away from the ramification points. It is smooth at a ramification point if and only if the corresponding zero of the quadratic differential $q_2$ is of order 1. Due to the specific type of characteristic equation the spectral curve comes with an involutive automorphism $\sigma: \Sigma \rightarrow \Sigma$ interchanging the sheets.  

The subset of quadratic differentials with simple zeros is an open and dense subset of $H^0(X,K^2)$, which we refer to as the regular locus. Its complement will be referred to as the singular locus and the fibres over the singular locus as singular (Hitchin) fibres. The Hitchin fibres over the regular locus form a dense subset of $\M$ fibred by complex tori.

\begin{theo}[Abelian Spectral Data 	\cite{Hi87a}]
Let $q_2 \in H^0(X,K^2)$ be a quadratic differential with simple zeros. Then $\H^{-1}(q_2)$ is a torsor over the Prym variety
\begin{align*} \mathsf{Prym}(\Sigma):=\{L \in \mathsf{Jac}(\Sigma)\mid L \otimes \sigma^*L=\O_X\}.
\end{align*}
\end{theo}   

\begin{proof}
This will be a special case of the description of Hitchin fibres with irreducible and reduced spectral curve given below. We want to sketch the classical construction for context. Let $\lambda=\eta\rest_{\Sigma}$ and $\Lambda=\div(\lambda)$.
Let $(E,\Phi) \in \H^{-1}(q_2)$, then $\lambda$ is an eigensection of $\pi^*\Phi$ and the line bundle of eigenvectors
\[ L = \Ker(\pi^*\Phi-\lambda \id_{\pi^*E}) 
\] is an element of the twisted Prym variety 
\[ \Pr_{\Lambda}(\Sigma)=\{L \in \Pic(\Sigma) \mid L \otimes \sigma^*L = \O(-\Lambda) \}
\] (see Theorem \ref{theo:stratifi}). We will see below that $\Pr_{\Lambda}(\Sigma)$ is a torsor over $\mathsf{Prym}(\Sigma)$ (Proposition \ref{prop:isom:prym}). The eigenline bundle uniquely determines the Higgs bundle by the algebraic pushforward $(E,\Phi)=\pi_*(L \otimes \pi^*K,\lambda)$ (cf. \ref{lem:2push}).
\end{proof}
\noindent Prym varieties are abelian varieties and were intensively studied in \cite{Mu74}. \\

In this work, we study Hitchin fibres with irreducible and reduced spectral curve. The spectral curve is irreducible and reduced if and only if $q_2$ has no global square root on $X$, i.e.\ there exists no $\lambda \in H^0(X,K)$, such that $\lambda^2=q_2$. In this case, there is a covering of Riemann surfaces associated to the characteristic equation. It is the unique two-sheeted branched covering of Riemann surfaces $\tilde{\pi}: \tilde{\Sigma} \rightarrow X$, such that there exists $\lambda \in H^0(\tilde{\Sigma},\tpst K)$ solving 
\[ \lambda^2 + \tpst q_2=0.
\] From a algebro-geometric perspective $\tilde{\Sigma}$ is the normalisation of $\Sigma$ and we will refer to $\tilde{\Sigma}$ as the normalised spectral curve. The geometry of this covering can be easily understood. The restriction
\[ \pi: \Sigma\setminus \pi^{-1}(Z(q_2)) \rightarrow X \setminus Z(q_2)
\] is an unbranched covering of Riemann surfaces and there is a unique way to extend it in a smooth way. Whenever the local polynomial equation for $\Sigma$ in a neighbourhood of $p \in \pi^{-1}(Z(q_2))$ is irreducible, or equivalently the corresponding zero of $q_2$ is of odd order, we glue in a disc, such that the covering map locally extends to $\tilde{\pi}: z \mapsto z^2$. If instead the local polynomial is reducible, or equivalently the zero of $q_2$ is of even order, we glue in two discs separating the two sheets. Hence, the branch points of $\tilde{\pi}: \tilde{\Sigma} \rightarrow X$ are the zeros of $q_2$ of odd order and with Riemann-Hurwitz the genus of $\tilde{\Sigma}$ is 
\[ g(\tilde{\Sigma})=2g-1+ \frac{n_\odd}{2},
\] where $n_\odd$ denotes the number of odd zeros of $q_2$ (without multiplicity).
Hitchin fibres with irreducible spectral curve have the following useful property:
\begin{lemm} Let $q_2$ be a quadratic differential with no global square root and $(E,\Phi) \in \H^{-1}(q_2)$. Then $(E,\Phi)$ is stable.
\end{lemm}
\begin{proof}
If there is a $\Phi$-invariant subbundle $M \subsetneq E$, then it is an eigenline bundle. Hence, the eigensections of $\Phi$ must exist on $X$ as global sections of $K$. This contradicts the assumption on $q_2$.
\end{proof}
In the following, we will also need to look at $M$-twisted $\SL(2,\C)$-Higgs bundles. These are pairs $(E,\Phi)$, where $E$ is a holomorphic vector bundle of rank 2, such that $\det(E)=\O_X$ and $\Phi \in H^0(X,\End(E) \otimes M)$ with $\tr(\Phi)=0$. Nitsure \cite{Ni91} constructed the moduli space of polystable $M$-twisted Higgs bundles in the algebraic category. It will be denoted by $\M(X,M)$. The Hitchin map is given by
\[ \H_M: \M(X,M) \rightarrow H^0(X,M^2), \quad (E,\Phi) \rightarrow \det(\Phi).
\] As above, the generic fibres can be identified with torsors over Prym varieties of dimension $g-1+\deg(M)$. However, the smooth locus of $\M(X,M)$ is no longer a holomorphic symplectic manifold and hence this torus fibration is no longer defined as a completely integrable system. If there exists a section $s \in H^0(X,MK^{-1})$, the smooth locus of $\M(X,M)$ has the structure of a Poisson manifold (see \cite{Ma94}).

\section{\texorpdfstring{$\sigma$}{sigma}-invariant Higgs Bundles on the Normalised Spectral Curve}\label{sec:pap}
\subsection{The pullback}
Let $p: Y \rightarrow X$ be a two-sheeted covering of Riemann surfaces and $\sigma: Y \rightarrow Y$ the involutive biholomorphism interchanging the sheets. 
\begin{defi}\label{def:sigma-inv} A $\sigma$-invariant holomorphic vector bundle $(E,\hat{\sigma})$ on $Y$ is holomorphic vector bundle $E$ on $Y$ with a lift
\[ \begin{tikzcd} E \arrow[d] \arrow[r,"\hat{\sigma}"] & E \arrow[d] \\
Y \arrow[r,"\sigma"] & Y
\end{tikzcd}
\] such that 
\begin{itemize}
\item[i)] $\hat{\sigma}^2=\id_E$, and
\item[ii)] $\hat{\sigma}\rest_{E_y}= \id_{E_y}$ for all ramification points $y \in Y$.
\end{itemize}
Let $(M,\hat{\sigma}_M)$ be $\sigma$-invariant holomorphic line bundle on $Y$. A $\sigma$-invariant $(M,\hat{\sigma}_M)$-twisted Higgs bundle $(E,\Phi,\hat{\sigma}_E)$ on $Y$ is a $M$-twisted Higgs bundle $(E,\Phi)$ on $Y$, such that $(E,\hat{\sigma}_E)$ is $\sigma$-invariant holomorphic vector bundle and
\begin{itemize}
\item[iii)] $(\hat{\sigma}_E \otimes \hat{\sigma}_M) \circ \Phi  = \Phi \circ \hat{\sigma}_E$.
\end{itemize}
\end{defi}

\begin{lemm} Let $(E,\Phi,\hat{\sigma}_E)$ be a $\sigma$-invariant $(M,\hat{\sigma}_M)$-twisted Higgs bundle and $g \in \A^0(Y,\SL(E))$ an element of the gauge group. Then $(gE, g \Phi g^{-1}, g \circ \hat{\sigma} \circ g^{-1})$ is a $\sigma$-invariant $(M,\hat{\sigma}_M)$-twisted Higgs bundle. 
\end{lemm}

Let $(M,\hat{\sigma}_M)$ a $\sigma$-invariant holomorphic line bundle on $Y$. Define 
\begin{align*} \mathcal{M}^{\sigma}(Y,M,\hat{\sigma}_M)=\left\{ (E,\Phi) \in \M_{\SL(2,\C)}(Y,M) \, \middle\vert \, \exists \hat{\sigma}: \begin{array}{l}
  (E,\Phi,\hat{\sigma}) \ \sigma\text{-invariant } \\ 
(M,\hat{\sigma}_M)\text{-twisted} 
\end{array} \right\}
\end{align*} 

\begin{prop}\label{prop:pull} \ 
\begin{itemize} \item[i)] Let $E$ be a holomorphic vector bundle on $X$. Then $p^*E$ has a induced lift $\hat{\sigma}_{p^*E}$, such that $(p^*E,\hat{\sigma}_{p^*E})$ is a $\sigma$-invariant holomorphic vector bundle.
\item[ii)] We have a natural map 
\[ p^*: \M_{\SL(2,\C)}(X,M) \rightarrow \M^{\sigma}(Y,p^*M,\hat{\sigma}_{p^*M}).
\] 
\end{itemize} 
\end{prop}
\begin{proof} \begin{itemize}
\item[i)] Let $U \subset X$ open, such that $E \rest_U \cong U \times \C^r$. The trivialisation induces a trivialisation $p^*E \rest_{p^{-1}(U)} \cong p^{-1}U \times \C^r$. If $x \in U$ is not a branch point, i.e.\ $p^{-1}(x)=\{ y, \sigma(y) \}$, such trivialisation induces a identification of the fibres $p^*E_y \cong p^*E_{\sigma(y)}$. This defines a lift $\hat{\sigma}_{p^*E}: p^*E \rightarrow p^*E$ away from the ramification points. This lift extends over the ramification points by the identity. Therefore, $(p^*E,\hat{\sigma}_{p^*E})$ is a $\sigma$-invariant holomorphic vector bundle.
\item[ii)] Clearly, $(p^*E,p^*\Phi) \in \M_{\SL(2,\C)}(Y,p^*M)$ and by i) $(p^*E,\hat{\sigma}_{p^*E})$ is a $\sigma$-invariant holomorphic vector bundle. Property iii) of Definition \ref{def:sigma-inv} becomes clear in a trivialisation as in the proof of i).
\end{itemize}
\end{proof}
\noindent In the sequel, a pullback will always carry the induced lift $\hat{\sigma}$ and we will omit it in the notation.

\subsection{The \texorpdfstring{$\sigma$}{sigma}-invariant pushforward}
\begin{defi}\label{def:sigma-inv.sheaf} Let $\xi$ be an analytic sheaf on $Y$. A lift $\hat{\sigma}: \xi \rightarrow \xi$ of $\sigma$ is a family of involutive homomorphisms of abelian groups
\[ \hat{\sigma}_V: H^0(V,\xi) \rightarrow H^0(\sigma(V),\xi)
\] commuting with restriction maps, such that for all $f \in \O_V$ and $s \in H^0(V,\xi)$
\[ \hat{\sigma}_V(fs)=(\sigma^*f) \hat{\sigma}_V(s).
\] The pair $(\xi,\hat{\sigma})$ is called an analytic $\sigma$-sheaf. 
\end{defi}

\begin{defi}\label{def:sigma:push} Let $(\xi,\hat{\sigma})$ be an analytic $\sigma$-sheaf on $Y$, then the $\sigma$-invariant pushforward $p_*(\xi,\hat{\sigma})$ is the analytic sheaf on $X$ defined through 
\[ H^0(U,p_*(\xi,\hat{\sigma}))=H^0(p^{-1}U,\xi)^{\hat{\sigma}}
\] for open sets $U \subset X$. Here $H^0(p^{-1}U,\xi)^{\hat{\sigma}}$ denotes the $\hat{\sigma}$-invariant sections of $(\xi,\hat{\sigma})$. 
\end{defi}

\begin{lemm}\label{lem:push:loc:free}\begin{itemize}
\item[i)] Let $(\xi,\hat{\sigma})$ be a locally free $\sigma$-sheaf of rank $r$ on $Y$, such that for every ramification point $y \in Y$ there exists an open, $\sigma$-invariant neighbourhood $V \subset Y$ of $y$ and an isomorphism $H^0(V,\xi)\cong \O_V^r$, such that
\[ \hat{\sigma}\rest_V:  \O_V^r  \rightarrow \O_V^r , \quad f \mapsto f\circ \sigma.
\] Then $p_*(\xi,\hat{\sigma})$ is locally free of rank $r$.
\item[ii)] Let $(E,\hat{\sigma})$ be a $\sigma$-invariant holomorphic vector bundle of rank $r$. Then $(\O(E),\hat{\sigma})$ satisfies the assumption in i). In particular, the pushforward $p_*(\O(E),\hat{\sigma})$ is locally free of rank $r$.
\end{itemize}
\end{lemm}
\begin{proof}
\begin{itemize}
\item[i)]Let $U \subset X$ an open subset trivializing the covering. Let $p^{-1}(U)=U_1 \cup U_{2}$. A section in $H^0(p^{-1}U,\xi)^{\hat{\sigma}}$ is fixed by its values on $U_1$. Hence $H^0(p^{-1}U,\xi)^{\hat{\sigma}}\cong \O_{U_1}^{r}  \cong \O_U^r$. Let $x \in X$ a branch point. By assumption there exists a neighbourhood $U \subset X$, such that 
\[ H^0(p^{-1}U,\xi)^{\hat{\sigma}}\cong \{ f \in \O_{p^{-1}U}^r \mid f= \sigma^*f \} \cong p^{-1}\O_U^{r} \cong \O_U^r.
\] 
\item[ii)] Clearly, a lift $\hat{\sigma}$ on $E$ induces a lift on the sheaf of sections $\hat{\sigma}: \O(E) \rightarrow \O(E)$ satisfying Definition \ref{def:sigma-inv.sheaf}. To check the extra assumption in i), let $y \in Y$ be a ramification point. Assumption ii) of Definition \ref{def:sigma-inv} guarantees the existence of a local frame of $\sigma$-invariant sections in a $\sigma$-invariant neighbourhood $V$ of $y$. Take a local basis for $E_{y}$ and extend it to a holomorphic frame $s_1,\dots,s_r$ of $E_{V}$. Then a $\sigma$-invariant frame is given by $s_1 + \hat{\sigma} s_1, \dots, s_r + \hat{\sigma}s_r$ for a small enough neighbourhood $V$ of $y$. A $\sigma$-invariant frame induces an isomorphism $\O(E)_V \cong \O_V^r$ such that $\hat{\sigma}\rest_V$ has the desired form.
\end{itemize}
\end{proof}

\begin{defi}
Let $(E,\hat{\sigma})$ be a $\sigma$-invariant vector bundle. We define the $\sigma$-invariant pushforward $p_*(E,\hat{\sigma})$ to be the vector bundle corresponding to the locally free sheaf $p_*(\O(E),\hat{\sigma})$.
\end{defi}

\begin{lemm}\label{lem:pup} Let $E$ be a holomorphic vector bundle on $X$ and $(p^*E,\hat{\sigma}_{p^*E})$ the corresponding $\sigma$-invariant holomorphic vector bundle on $Y$. Then
\[ p_*(p^*E,\hat{\sigma}_{p^*E})=E.
\]
\end{lemm}

\begin{exam}\label{exam:I} Let $p: Y \rightarrow X$ be an unbranched 2-covering of Riemann surfaces. Let $L$ be a line bundle on $X$ and $(p^*L,\hat{\sigma})$ the induced $\sigma$-invariant line bundle on $Y$. Then $-\hat{\sigma}$ is another lift of $\sigma$ on $L$. However, $p_*(p^*L,-\hat{\sigma}) \ncong L$. We have 
\[ p_*(p^*L,-\hat{\sigma}) \cong L \otimes I,
\] where $I=p_*(\O_Y,-\id_{\O_Y})$ is the unique non trivial line bundle on $X$, which pulls back to the trivial bundle on $Y$. Then $p^*(I^2) \cong \O_Y$ and the induced lift $\hat{\sigma}_{p^*(I^2)}$ is the identity. Hence, $I^2=\O_X$. 
\end{exam}

\subsection{Pullback and pushforward of singular Hitchin fibres}
Let $q_2 \in H^0(X,K^2)$ be a quadratic differential with no global square root on $X$. Let $\tilde{\pi}: \tilde{\Sigma} \rightarrow X$ be the covering by the normalized spectral curve and $\sigma: \tilde{\Sigma}\rightarrow \tilde{\Sigma}$ the involution interchanging the sheets. We want to parametrize the singular Hitchin fibres by parametrizing their pullback to $\tilde{\Sigma}$. However, the pullback
\[ \tpst: \H^{-1}(q_2) \rightarrow \mathcal{M}^{\sigma}(\tilde{\Sigma},\tpst K)
\] is not injective in general, because there can be multiple lifts of $\sigma$.

\begin{exam}\label{exam:non:inj}
Let $q_2 \in H^0(X,K^2)$ be a quadratic differential with only double zeros, which has no global square root on $X$. Then $\tilde{\Sigma} \rightarrow X$ is a 2-sheeted unbranched covering of Riemann surfaces. In this case, there exists a non-trivial line bundle $I$ with $\tilde{\pi}^*(I)\cong \O_{\tilde{\Sigma}}$ and $I^2=\O_X$ (see Example \ref{exam:I}). For $(E,\Phi) \in \H^{-1}(q_2)$, also $(E\otimes I,\Phi)\in\H^{-1}(q_2)$. But we clearly have 
\[ \tilde{\pi}^*(E,\Phi) \cong \tilde{\pi}^*(E\otimes I,\Phi).\]
\end{exam}

\begin{prop}\label{lem:lift}
Let $q_2 \in H^0(X,K^2)$ be a quadratic differential with no global square root. Let \[ 
(E,\Phi) \in \tilde{\pi}^* \H^{-1}(q_2) \subset \mathcal{M}^{\sigma}(\tilde{\Sigma},\tpst K).
\]
\begin{itemize}
\item[i)] If $q_2$ has at least one zero of odd order, then there is a unique lift $\hat{\sigma}$, such that $(E,\Phi,\hat{\sigma})$ is a $\sigma$-invariant Higgs bundle. 
\item[ii)] If $q_2$ has only zeros of even order, then there are two such lifts $\pm\hat{\sigma}$.
\end{itemize}
\end{prop}
\begin{proof} Let $(E,\Phi) \in \tilde{\pi}^* \H^{-1}(q_2)$. Assume that there a two lifts $\hat{\sigma}_1,\hat{\sigma}_2$, such that $(E,\Phi,\hat{\sigma}_i)$ is a $\sigma$-invariant Higgs bundle. Then $\hat{\sigma}_1 \circ \hat{\sigma}_2 \in \Aut(E,\Phi)$. If $(E,\Phi)$ is stable, this implies that $\hat{\sigma}_1=\pm \hat{\sigma}_2$. If in addition, $q_2$ has only even zeros, the normalised spectral covering $\tilde{\pi}$ is unbranched and this gives the two possible lifts. If $(E,\Phi)$ is stable, and $q_2$ has at least on zero of odd order, $\tilde{\pi}: \tilde{\Sigma} \rightarrow X$ has at least one ramification point $p \in \tilde{\Sigma}$. In particular, $(\hat{\sigma}_1)_p=(\hat{\sigma}_2)_p=\id_{E_p}$ and therefore $\hat{\sigma}_1= \hat{\sigma}_2$. On the other hand, $(E,\Phi) \in \tpst \H^{-1}(q_2)$ is strictly polystable if and only if
\[ (E,\Phi)= \left(M \oplus M^{-1}, \begin{pmatrix} \lambda & 0 \\ 0 & -\lambda \end{pmatrix}\right).
\] with $M \in \Jac(\tilde{\Sigma})$. Hence, $q_2$ has only even zeros and $\hat{\sigma}_1=g \hat{\sigma}_2$ with
\[ g \in \Aut(E,\Phi)=\left\{ \begin{pmatrix} t & 0 \\ 0 & t^{-1} \end{pmatrix} \mid t \in \C^*\right\},
\] such that $g^2=\id_E$. Hence again, $g=\pm \id_E$.
\end{proof}

\begin{prop}\label{prop_pi*_inj}
Let $q_2 \in H^0(X,K^2)$ be a quadratic differential with no global square root. The pullback
\[ \tilde{\pi}^*: \H^{-1}(q_2) \rightarrow \mathcal{M}^{\sigma}(\tilde{\Sigma},\tpst K)
\] 
\begin{itemize}
\item[i)] is injective, if $q_2$ has at least one zero of odd order, and
\item[ii)] is generically two-to-one, if $q_2$ has only even zeros. 
\end{itemize}Let $I$ be the unique non-trivial line bundle with $\tpst I= \O_Y$. Then the non-injectivity in ii) is due to the identification of $(E,\Phi)$ and $(E\otimes I,\Phi)$ by pullback. 
\end{prop}
\begin{proof}
We already saw in Lemma \ref{lem:lift} that in the first case there is a unique lift $\hat{\sigma}$. Hence the injectivity follows from Lemma \ref{lem:pup}. In the second case, we saw that there are two possible lifts $\pm \hat{\sigma}$. From Example \ref{exam:I} this implies
\[ \pi_*(E,\hat{\sigma})=(\pi_*(E,-\hat{\sigma})) \otimes I.
\] 
Together with Lemma \ref{lem:pup}, this gives the result in case ii).
\end{proof}

\begin{exam}\label{exam:branchpoint} In case ii) branching exists. The section $\lambda: \tilde{\Sigma} \rightarrow \tpst K$ has the property $\sigma^*\lambda=-\lambda$. Hence, it descends to a section $\alpha \in H^0(X,KI)$. Then
\[ (E,\Phi)= \left( I^\frac12 \oplus I^{-\frac{1}{2}}, \begin{pmatrix}
 0 & \alpha \\ \alpha & 0 \end{pmatrix} \right)
\] defines a Higgs bundle in $\H^{-1}(q_2)$, such that $E\otimes I \cong E$.

\end{exam}

\section{Hecke Transformations}\label{sec:Hecke}
In Section \ref{sec:strat}, we will stratify singular Hitchin fibres by fibre bundles over twisted Prym varieties. The twisted Prym variety will parametrize the eigenline bundles of the Higgs bundles in the stratum. The fibres of these bundles parametrize the manipulation of Higgs bundles by Hecke transformations. In this section, we recall the definition of Hecke transformation, originally introduced in \cite{NaRa75}, and adapt it to our purpose. It is a special case of the more general concept of Hecke modifications, see \cite{Ba10, Wo13,HeWa19}. For simplicity, we will only treat the case of holomorphic vector bundles of rank 2. \\

Let us start by recalling the rank 1 analogue. The Hecke transformation of a line bundle $L$ at $p \in X$ is the line bundle $L(-p)$. We have an exact sequence 
\[ 0 \rightarrow \O(L(-p)) \xrightarrow{s_p} \O(L) \rightarrow \mathcal{T}_X(p) \rightarrow 0,
\] where $s_p$ is a canonical section of $\O(p)$ and $\mathcal{T}_X(p)$ is the torsion sheaf with a stalk of length 1 at $p$. 

\begin{defi}[\cite{NaRa75},\cite{HwRa04}]\label{def:Hecke}
Let $E$ be a holomorphic vector bundle of rank $2$ on a Riemann surface $X$. Let $p \in X$ and $\alpha \in E_p^\vee\setminus\{0\}$, the dual fibre at $p$. The Hecke transformation $\hat{E}^{(p,\alpha)}$ of $E$ is defined through the exact sequence of coherent sheaves 
\[ 0 \rightarrow \O(\hat{E}^{(p,\alpha)}) \rightarrow \O(E) \xrightarrow{\alpha} \mathcal{T}_X(p) \rightarrow 0.
\]
\end{defi}

For a more concrete description of Hecke transformations, let us to describe them on the level of transition functions. Let $\mathcal{GL}(n)$ denote the sheaf of holomorphic $\mathsf{GL}(n,\C)$-valued functions on $X$. Let $\mathcal{U}=\{ U_i\}_{i=1}^{m}$ a covering of $X$ by contractible open sets, such that $p \in U_i$ if and only if $i=1$. Let $\{ \psi_{ij}\} \in \check{H}^1(\mathcal{U},\mathcal{GL}(2))$ transition functions for $E$. Choose a holomorphic frame $s_1,s_2$ of $E \rest_{U_1}$, such that $\alpha=(s_2)_p^\vee$. Define a covering $\mathcal{V}=\{V_i\}_{i=0}^m$ by $V_0=U_1$, $V_1=U_1\setminus\{p\}$ and $V_i=U_i$ for $i \geq 2$. Define transition functions $\{ \hat{\psi}_{ij}\} \in \check{H}^1(\mathcal{V},\mathcal{GL}(2))$ by
\begin{align}\label{equ:trans:Hecke1}
 \hat{\psi}_{01}: V_0 \cap V_1 \times \C^2 &\rightarrow V_0 \cap V_1 \times \C^2, \\
 (z,x_1,x_2) \quad &\mapsto \quad (z,x_1,zx_2) \notag
 \end{align} with respect to the frame $s_1,s_2$,
 \begin{align}\label{equ:trans:Hecke2}
 \hat{\psi}_{0j}= \psi_{1j} \circ \hat{\psi}_{01}, \quad \hat{\psi}_{j0}=\hat{\psi}_{0j}^{-1} \quad \text{for}  \quad j \geq 1, \quad \text{and} \quad \hat{\psi}_{ij}=\psi_{ij} \quad \text{for} \quad i,j \geq 1. 
\end{align} 
\begin{lemm} The holomorphic vector bundle associated to the transition functions $\{\hat{\psi}_{ij}\}\in \check{H}^1(\mathcal{V},\mathcal{GL}(2))$ is the Hecke transformation $\hat{E}^{p,\alpha}$ of $E$.  
\end{lemm}
\begin{proof}
By definition of the transition function $\hat{\psi}_{01}$, the associated vector bundle fits into an exact sequence as in Definition \ref{def:Hecke}.
\end{proof}

We generalize this concept by allowing higher order twists. Let $\Div^+(X)$ denote the set of effective divisors on $X$. Let $D \in \Div^+(X)$ and $E$ a holomorphic vector bundle on $X$. Hecke transformations at $D$ will be parametrised by polynomial germs on $D$. Define
\[ H^0(D,E):=\bigoplus\limits_{p \in \supp D} \O(E)_p / \sim,
\] where $[s_1] \sim [s_2]$ if and only if $\ord_p([s_1]-[s_2]) \geq D_p$, for all $p \in \supp D$. Furthermore, denote by $H^0(D,E)^* \subset  H^0(D,E)$ the equivalence classes of germs, such that for all $p \in \supp D$ the evaluation at $p$ is non-zero. 
\begin{defi}
Let $E$ be a holomorphic vector bundle of rank 2. Let $D \in \Div^+(X)$ and $\alpha \in H^0(D,E^\vee)^*$. Then the Hecke transformation $\hat{E}^{(D,\alpha)}$ of $E$ at $D$ in direction $\alpha$ is defined by the exact sequence of locally free sheaves
\[ 0 \rightarrow \O(\hat{E}^{(D,\alpha)}) \rightarrow \O(E) \xrightarrow{\alpha} \mathcal{T}_X(D) \rightarrow 0,
\] where $\mathcal{T}_X(D)$ is the torsion sheaf of length $D_p$ at $p \in \supp D$. 
\end{defi}

\begin{lemm} Let $D \in \Div^+(X)$ and $\alpha \in H^0(D,E^\vee)^*$. Then
$\det(\hat{E}^{(D,\alpha)})= \det(E)(-D)$.
\end{lemm}
\begin{proof}
By definition, $\det(\mathcal{T}_X(D)) \cong \O(D)$.
\end{proof}

\noindent For our purposes, it will be more convenient to use the dual version of this concept.
\begin{defi}\label{defi:dual:Hecke}
Let $D \in \Div^+(X)$ and $\alpha \in H^0(D,E)^*$. Then the (dual) Hecke transformations $\hat{E}^{(D,\alpha)}$ of $E$ at $D$ in direction $\alpha$ is defined by
\[ (\hat{E}^{(D,\alpha)})^\vee:=\widehat{E^\vee}^{(D,\alpha)}.
\]
\end{defi}

\begin{lemm}\label{lemm:trans:Hecke} Let $\{ \psi_{ij}\} \in \check{H}^1(\mathcal{U},\mathcal{GL}(2))$ transition functions of $E$ as above. For $p \in X, l \in \N$, let $D:=lp \in \Div^+(X)$. Let further $\alpha \in H^0(D,E)^*$. The Hecke transformation $\hat{E}^{(D,\alpha)}$ is the holomorphic vector bundle associated to the transition functions $\{ \hat{\psi}_{ij}\}\in \check{H}^1(\mathcal{V},\mathcal{GL}(2))$ defined as in (\ref{equ:trans:Hecke1}),(\ref{equ:trans:Hecke2}), where the frame $s_1,s_2$ is chosen, such that 
\[ [(s_2)_p]=\alpha\in H^0(D,E)
\] and 
\begin{align*}\hat{\psi}_{01}: V_0 \cap V_1 \times \C^2 &\rightarrow V_0 \cap V_1 \times \C^2 \\
(z,x_1,x_2) \quad &\mapsto \quad (z,x_1,z^{-l}x_2).
\end{align*} More generally, for $D \in \Div^+(X)$ and $\alpha \in H^0(D,E)^*$, we obtain transition functions of $\hat{E}^{(D,\alpha)}$ by introducing a new transition function like this for all $p \in \supp(D)$.
\end{lemm}

\begin{lemm} \label{lemm:det:Hecke} Let $D \in \Div^+(X)$ and $\alpha \in H^0(D,E)^*$. Then
$\det(\hat{E}^{(D,\alpha)})= \det(E)(D)$.
\end{lemm}

\subsection{Parameters of Hecke transformations}
\begin{lemm} Let $D \in \Div^+(X), \alpha \in H^0(D,E)^*$ and $\phi \in H^0(D,\O_X)^*$. Then 
\[ E^{(D,\alpha)} \cong E^{(D,\phi \alpha)}.
\]
\end{lemm}

\begin{prop}\label{prop:G} $H^0(D,\O_X)^*$ is a complex solvable Lie group with respect to the multiplication of germs of non-vanishing holomorphic functions. Let $D=lp$ with $l \in \N$ and $p \in X$. Then
\[ H^0(D,\O_X)^* \cong \left\{ \begin{pmatrix} x_0 & x_1 & \dots & x_{l-1} \\ 
& \ddots & \ddots & \vdots \\ 
& & x_0 & x_1 \\ & & & x_0
\end{pmatrix} \mid x_0 \in \C^*, x_i \in \C \right\}.
\] For $D \in \Div^+(X)$, $H^0(D,\O_X)^*$ is isomorphic to a Cartesian product of such groups, one for each $p \in \supp(D)$.
\end{prop}

\noindent An equivalence class in the quotient $H^0(D,E)/ H^0(D,\O_X)^*$ is referred to as a Hecke parameter.

\subsection{Guiding example}
As a guiding example, we show how the algebraic pushforward of a line bundle along a two-sheeted covering of Riemann surfaces can be recovered using Hecke transformations and the $\sigma$-invariant pushforward defined in Section \ref{sec:pap}.

Let $p: Y \rightarrow X$ be a two-sheeted covering of Riemann surfaces and $\sigma: Y \rightarrow Y$ the holomorphic involution interchanging the sheets. Denote by $R \subset Y$ the ramification divisor.
Let $L \in \Pic(Y)$, then $E=L \oplus \sigma^*L$ has a natural lift $\hat{\sigma}: E \rightarrow E$ induced by pullback along $\sigma$. At a ramification point, we can choose a frame, such that $\hat{\sigma}$ is locally given by 
\[ \begin{pmatrix} 0 & 1 \\ 1 & 0\end{pmatrix}.
\] Hence, $E$ is no $\sigma$-invariant holomorphic vector bundle (cf. Definition \ref{def:sigma-inv}). This can be corrected by applying a Hecke transformation. 

Choose a neighbourhood $U$ of $\Fix(\sigma)$ separating all ramification points and a frame $s \in H^0(U,L)$. Then 
\[ s_1 = s \oplus \sigma^*s, \quad s_2=s \oplus - \sigma^*s
\] is a frame of $E$ diagonalizing $\hat{\sigma}$. Let
\[ \alpha=[s_2]_R^\vee \in H^0(R,E^\vee)^*.
\] For $y \in \supp(R)$ choose a coordinate $z$, such that the involution is given by $\sigma: z \mapsto -z$. We saw above that $\hat{E}^{(R,\alpha)}$ is obtained from $E$ by introducing new transition functions of the form 
\begin{align*} &\hat{\psi}_{01}: V_0 \cap V_1 \times \C^2 \rightarrow V_0 \cap V_1 \times \C^2  \\
  &\qquad \quad (z,x_1,x_2) \quad \mapsto \quad (z,x_1,zx_2)
\end{align*} at every point $y \in \supp(R)=\Fix(\sigma)$. Here $\hat{\sigma}$ induces a lift of $\sigma$ on $\hat{E}^{(R,\alpha)}$, that we keep calling $\hat{\sigma}$. The frame $s_1,zs_2$ extends to a $\hat{\sigma}$-invariant frame $s_1^\sigma,s_2^\sigma$ of $\hat{E}^{(R,\alpha)}$. Hence, $(\hat{E}^{(R,\alpha)},\hat{\sigma})$ is a $\sigma$-invariant holomorphic vector bundle and $p_*(\hat{E}^{(R,\alpha)},\hat{\sigma})$ defines a holomorphic vector bundle of rank 2 on $X$. 

\begin{lemm}\label{lem:2push} $p_*(\hat{E}^{(R,\alpha)},\hat{\sigma})=p_*L$.
\end{lemm}

\begin{proof}
Let $U_1 \subset X$ be open, contractible subset trivializing the covering $p$, i.e.\ $p^{-1}U_1= V^+ \sqcup V^-$. Then $\O(p_*L)$ is free of rank 2 over $\O_{U_1}$. This is apparent from decomposing 
\[ H^0(U_1,p_*L)=H^0(p^{-1}U_1,L)=H^0(V^+,L) \oplus H^0(V^-,L).
\] Hence, we have a natural isomorphism
\begin{align} H^0(U_1,p_*L)\cong H^0(U_1,p_*(\hat{E}^{(D,\alpha)},\hat{\sigma}))= H^0(p^{-1}U_1,L \oplus \sigma^*L)^{\hat{\sigma}}. \label{iso1}
\end{align} Let $U_2 \subset X$ be open, contractible neighbourhood of a branch point $x \in X$. Choose a coordinate on $p^{-1}(U_2)$, such that $\sigma\rest_{p^{-1}(U_2)}: z \mapsto -z$. Let $s \in H^0(p^{-1}U_2,L)$ a local frame and $\phi \in H^0(p^{-1}U_2,L)$. Then there exist $\phi_1,\phi_2 \in \O_{U_2}$, such that
\begin{align} \phi(z)= \phi_1(z^2)s+\phi_2(z^2)zs. \label{phi}
\end{align} Hence, $p_*L\rest_{p^{-1}(U_2)}$ is free over $\O_{U_2}$ of rank 2 with generators $s,zs$. Let $s_1,s_2$ be the $\sigma$-invariant frame of $\hat{E}^{(D,\alpha)}$ defined above. Then we define an isomorphism
\begin{align} H^0(p^{-1}U_2,L) \rightarrow H^0(p^{-1}U_2,\hat{E}^{(D,\alpha)})^{\hat{\sigma}}, \qquad
\phi \mapsto \phi_1s_1 + \phi_2s_2. \label{iso2}
\end{align}
We claim that (\ref{iso1}) and (\ref{iso2}) define an isomorphism of locally free sheaves, i.e.\ they commute with the restriction functions. 

Let $U_1,U_2 \subset X$ as above, such that $U_1 \subset U_2$. Choosing a coordinate $w$ on $U_1$ we can identify the two branches $V^\pm$ with the square roots $\pm \sqrt{w}$. Let $\phi \in H^0(U_2,p_*L)=H^0(p^{-1}U_2,L)$. From (\ref{phi}) we obtain
\begin{align*} \phi\rest_{V^+}&= (\phi_1(z^2)+ \phi_2(z^2)z) s\rest_{V^+}= (\phi_1(w)+ \phi_2(w)\sqrt{w})s\rest_{V^+}, \\
\phi\rest_{V^-}&= 
(\phi_1(z^2)+ \phi_2(z^2)z)s \rest_{V^-} = (\phi_1(w)- \phi_2(w)\sqrt{w})s \rest_{V^-}.
\end{align*}
So the restriction map is given by 
\begin{align*} 
r_{U_2U_1}=\begin{pmatrix} 1 & \sqrt{w} \\ 1 & -\sqrt{w}\end{pmatrix}.
\end{align*}
This agrees with the restriction map of $p_*(\hat{E}^{(D,\alpha)},\hat{\sigma})$ by construction.
\end{proof}

\begin{coro}\label{coro:push:trivial} Consider a two-sheeted covering of Riemann surfaces $p: Y \rightarrow X$. Then
\[ p_*\O_Y=\O_X \oplus J.
\] If $p$ is not unbranched, then $J \in \Pic(X)$ is the unique line bundle, such that $p^*J=\O(-R)$, where $R$ is the ramification divisor of $p$. If $p$ is unbranched, then $J \in \Jac(X)$ is the unique non-trivial line bundle, such that $p^*J=\O_X$.
\end{coro}
\begin{proof}
Let $L=\O_Y$ in the construction above. So, $E=\O_Y \oplus \O_Y$ and 
\[ \hat{\sigma}=\begin{pmatrix}0 & 1 \\ 1 & 0 \end{pmatrix}
\] The diagonalizing frame for $\hat{\sigma}$ defines a global splitting
\[E= \O_Y \begin{pmatrix} 1 \\ 1 \end{pmatrix} \oplus \O_Y\begin{pmatrix} 1 \\ -1 \end{pmatrix} \quad \text{with} \quad \hat{\sigma}= \begin{pmatrix} 1 & 0 \\ 0 & -1 \end{pmatrix}.
\] If $p$ is a branched covering, we apply a Hecke transformation and obtain  
\[ \hat{E}^{(R,\alpha)} = \O_Y \oplus \O_Y(-R)=p^*\O_X \oplus p^*J.
\] 
The uniqueness of $J$ follows from the injectivity of the pullback along bra-\linebreak nched coverings. If $p$ is unbranched, $E$ is a $\sigma$-invariant vector bundle with the lifted $\sigma$-action on the second factor being $-\id_{\O_Y}$. Hence, the second factor descends to the line bundle $J$. In both cases, Lemma \ref{lem:2push} gives the result. 
\end{proof}

\section{Moduli of \texorpdfstring{$\sigma$}{sigma}-Invariant Higgs Bundles}\label{sec:strat}
After identifying the Hitchin fibres with moduli spaces of $\sigma$-invariant Higgs bundles on the normalised spectral curve in Section \ref{sec:pap}, we will now prove the stratification result for these moduli spaces in a more general setting. Thereafter, we will identify these strata as fibre bundles over Prym varieties. We return to the case of singular $\SL(2,\C)$-Hitchin fibres in the following section.  

\subsection{The stratification}

Let $Y$ be a Riemann surface of genus $g(Y) \geq 2$ with an involutive biholomorphism $\sigma: Y \rightarrow Y$. Denote by $Y/\sigma$ the unique Riemann surface, such that there exists a branched two-sheeted covering of Riemann surfaces $p: Y \rightarrow Y/\sigma$ factoring through $\sigma$. In particular, the genus of $Y/\sigma$ is given by
\[ g(Y/\sigma)= \tfrac{1}{2}\left(g(Y)+1-\tfrac{1}{2}\#\Fix(\sigma)\right).
\] Depending on $g(Y)$, this restricts the number of fix points of the involution $\sigma$. Let $M$ be a line bundle on $Y/\sigma$ with a non-zero section $\lambda: Y \rightarrow p^*M$, such that $\hat{\sigma}\lambda=-\lambda$. Here $p^*M$ is regarded as a $\sigma$-invariant holomorphic line bundle with the lift $\hat{\sigma}$ induced by pullback (cf. Proposition \ref{prop:pull}). In particular, $\lambda$ has a zero of odd order at every branch point. Let $\Lambda=\div(\lambda)$. In this section, we parametrize
\[ \M_\lambda^\sigma = \M^\sigma(Y,p^*M,\lambda):= \M^\sigma(Y,p^*M) \cap \H_{p^*M}^{-1}(-\lambda^2),
\] the polystable $\sigma$-invariant $p^*M$-twisted $\SL(2,\C)$-Higgs bundles on $Y$ with characteristic equation
\[ (T+\lambda)(T-\lambda).
\] By assumption, $-\lambda^2$ is a $\sigma$-invariant section of $p^*M^2$ and hence descends to $a \in H^0(Y/\sigma,M^2)$. Then $\M_\lambda^\sigma$ is identified with the image of $p^*: \H_M^{-1}(a) \rightarrow \M(Y,p^*M)$ by Proposition \ref{prop:pull} and is therefore an analytic subspace. 

\begin{lemm}\label{lem:loc:form} Let $(E,\Phi) \in \M_\lambda^\sigma$. Let $y \in Y$ and $m \in H^0(U, p^*M)$ be a local frame in an open neighbourhood $U$ of $y$. There exists a local frame of $E\rest_U$, such that the Higgs field is given by
\[ \Phi=z^{D_y} \begin{pmatrix} 0& 1 \\ z^{2\Lambda_y-2D_y} & 0 \end{pmatrix} \otimes m.
\] 
\end{lemm}
\begin{proof}
Choose a coordinate disc $(U,z)$ centred at $y$, such that the determinant $\det(\Phi)=-z^{2\Lambda_y} m^2$. There exists a non vanishing section $\phi \in H^0(U,\End(E))$ such that
\[ \Phi(z)=z^{D_y}\phi(z)m. \]
There are two possible Jordan forms of $\phi$ at $y$. If $D_y < \Lambda_y$ there is one Jordan block of size 2, if $D_y = \Lambda_y$, $\phi$ is diagonalizable with eigenvalues $\pm 1$. Thus, after a constant gauge transformation we can assume
\[ \phi(z)=\begin{pmatrix} a(z) & b(z) \\ c(z) & -a(z) \end{pmatrix} \quad \text{with} \quad \phi(0)=\begin{pmatrix}
0 & 1 \\ \ast & 0 \end{pmatrix} \,.
\] Hence,
\[ g= \frac{1}{\sqrt{b(z)}}\begin{pmatrix} b(z) & 0 \\ -a(z) & 1 \end{pmatrix}\in \SL(E\rest_{U}), 
\] is a well-defined gauge transformation, such that 
\[ g^{-1} \phi g = \begin{pmatrix} 0 & 1 \\ -\det(\phi) & 0\end{pmatrix}=\begin{pmatrix} 0 & 1 \\ z^{2\Lambda_y-2D_y} & 0\end{pmatrix}.
\] 
\end{proof}

For $(E,\Phi) \in \M_\lambda^\sigma$, we denote by $\div(\Phi)$ the vanishing divisor of $\Phi$. The properties of such divisors are summarized in the following definition. 

\begin{defi} 
An effective divisor $D \in \Div(Y)$ is called $\sigma$-Higgs divisor on $(Y,\sigma,\lambda)$, if $0 \leq D \leq \Lambda $, $\sigma^*D=D$ and for all $y \in \mathsf{Fix}(\sigma) \subset \mathsf{supp}\Lambda$
\[ D_y \equiv 0 \mod 2.
\]
\end{defi}

\begin{theo} \label{theo:stratifi}
There exists a stratification 
\[ \M^\sigma(Y,p^*M,\lambda)= \bigsqcup\limits_D \mathcal{S}_D
\] by locally closed analytic subsets 
\[ \mathcal{S}_D=\{ (E,\Phi) \in \M^\sigma(Y,p^*M,\lambda) \mid \div(\Phi)=D\}
\] indexed by $\sigma$-Higgs divisors $D \in \Div(Y)$. 
\end{theo}
\begin{proof}
First, it is easy to see that for $(E,\Phi) \in \M^\sigma(Y,p^*M,\lambda)$ the vanishing divisor $\div(\Phi)$ is a $\sigma$-Higgs divisor. These divisors form a lower semi-continuous invariant on $\M_\lambda^\sigma$ (cf. Lemma \ref{lem:loc:form}). In particular, for fixed $\sigma$-Higgs divisor $D$
\[ \bigcup\limits_{D' \geq D} \mathcal{S}_{D'} \quad \mathrm{is\ closed\quad and} \quad \bigcup\limits_{D' \leq D} \mathcal{S}_{D'} \quad \mathrm{is\ open.}
\] Hence, $\mathcal{S}_D$ is locally closed. To see that the closed subset is an analytic subset, we need to identify it as the Hitchin  fibre of $\sigma$-invariant $\SL(2,\C)$-Higgs bundles with a different twist. Fix a $\sigma$-Higgs divisor $D$ and let $s_D$ be the canonical section of $\O(D)$, which is $\sigma$-invariant. Then $(p^*M)(-D)=p^*(M(-\tfrac12 \mathsf{Nm} D))$ and hence it is a pullback. Moreover $\tfrac{\lambda}{s_D} \in H^0(p^*M(-D))$ satisfies $\sigma^*(\lambda/s_D)=-\lambda/s_D$. So  
\[ \M^\sigma(Y,p^*M(-D),\lambda/s_D)
\] is the subspace of $\sigma$-invariant $p^*M(-D)$-twisted $\SL(2,\C)$-Higgs bundles on $Y$ with determinant $-\lambda^2/s_D^2$. This is the pullback of a Hitchin fibre in the moduli space of $M(-\tfrac12 \mathsf{Nm} D)$-twisted $\SL(2,\C)$-Higgs bundles on $Y/\sigma$ and hence it is an analytic subspace of $\M(Y,p^*M(-D))$. There is a holomorphic bijective map
\begin{align*} \M^\sigma(Y,p^*M(-D),\lambda/s_D) &\rightarrow \bigcup\limits_{D' \geq D} \mathcal{S}_{D'} \subset \M^\sigma(Y,p^*M,\lambda), \\
(E,\Phi) \qquad &\mapsto \quad (E,s_D \Phi).
\end{align*}
Therefore, its image is an analytic subspace of $\M^\sigma(Y,p^*M,\lambda)$ (see \cite{GPR94} I.10.13).
\end{proof}

\subsection{Prym varieties}

\begin{defi} Let $D$ be an effective divisor on $Y$, such that $\sigma^*D=D$. Then the $D$-twisted Prym variety of $(Y,\sigma)$ is defined by
\[ \Pr_D=\{ L \in \mathsf{Pic}(Y) \mid L \otimes \sigma^*L = \O_Y(D)^{-1} \}.
\]
\end{defi}

\begin{theo} Consider $\M_\lambda^\sigma$ as above. For every stratum $\mathcal{S}_D$, there exists a holomorphic map 
\[ \mathsf{Eig}_D: \mathcal{S}_D \rightarrow \Pr_{\Lambda-D}, \qquad (E,\Phi) \mapsto \Ker(\Phi-\lambda\id_E).
\]
\end{theo}
\begin{proof}
Let $(E,\Phi) \in \mathcal{S}_D$ and let $\O(L)=\Ker(\Phi-\lambda\id_E)$ the sheaf-theoretical kernel. Then $\O(\sigma^*L)= \linebreak \Ker(\Phi+\lambda\id_E)$. The inclusions $\O(L) \rightarrow \O(E),  \O(\sigma^*L) \rightarrow \O(E)$ define an exact sequence of coherent analytic sheaves
\[ 0 \rightarrow \O(L) \oplus \O(\sigma^*L) \rightarrow \O(E) \rightarrow \mathcal{T} \rightarrow 0,
\] where $\mathcal{T}$ is a torsion sheaf supported at $Z(\lambda)$. This torsion sheaf can be constructed explicitly using the local description of $\Phi$ in a neighbourhood of $p \in Z(\lambda)$ given in Lemma \ref{lem:loc:form}. In particular,
\[ \O_Y= \det(E)=L \otimes \sigma^*L \otimes \det(\mathcal{T})=L \otimes \sigma^*L \otimes \O(\Lambda-D).
\]
\end{proof}

\begin{prop}\label{prop:isom:prym}
Let $D$ be a $\sigma$-Higgs divisor on $(Y,\sigma,\lambda)$. The twisted Prym variety $\Pr_{\Lambda-D}$ is an abelian torsor of dimension 
\[ g(Y)-g(Y/\sigma) = \tfrac{1}{2}\left(g(Y)-1+\tfrac{1}{2} \# \Fix(\sigma)\right).
\] over
\begin{itemize}
\item[i)] $\Prym(\tilde{\Sigma}):=\Nm^{-1}(\O_{Y/\sigma})$, if $\Fix(\sigma) \neq \varnothing$,
\item[ii)]$\Pr_0=\Nm^{-1}(\O_{Y/\sigma}) \sqcup \Nm^{-1}(I)$, if $\Fix(\sigma) = \varnothing$, where $I$ is the unique non-trivial line bundle $I \in \ker(p^*)$.
\end{itemize} 
\end{prop}
\begin{proof} 
Let $N={\Lambda-D}$. We need to show that $\Pr_N$ is non-empty. If there exists $L \in \Pr_N$, the group action
\begin{align*}
\Pr_0 \rightarrow \Pr_N, \qquad M \mapsto L \otimes M
\end{align*} is simply transitive. Hence, $\Pr_N$ is a torsor over $\Pr_0$ in both cases. Due to the map $\mathsf{Eig}_D: S_D \rightarrow \Pr_{N}$, it is enough to show that $\mathcal{S}_D$ is non-empty for $\sigma$-Higgs divisors $D$. After identifying $\M_\lambda^\sigma$ with the $\SL(2,\C)$-Higgs bundles on $Y/\sigma$ in the corresponding Hitchin fibre it will be easy to construct examples with all possible $\sigma$-Higgs divisors by looking at $\SL(2,\R)$-Higgs bundles (see Proposition \ref{theo:realpts:even}). 

We are left with identifying $\Pr_0$. Let $L \in \Pic(Y)$, we have $(p^* \circ \Nm)(L) = L \otimes \sigma^*L$. If $\Fix(\sigma) \neq \varnothing$, then $p$ has a branch point and $p^*$ is injective. This gives i). If $\Fix(\sigma) = \varnothing$, then $\Ker(p^*)=\{ \O_{Y/\sigma},I \}$ and hence we obtain ii). 
\end{proof}

\subsection{Hecke parameters}

The fibres of each stratum $\S_D$ over $\Pr_{\Lambda-D}(Y)$ can be identified as Hecke parameters (cf. Section 3). Consider $\M_\lambda^\sigma$ as above, $D$ a $\sigma$-Higgs divisor and $L \in \Pr_{\Lambda-D}(Y)$. Let 
\[ (E_L,\Phi_L)=\left(L \oplus \sigma^*L, \begin{pmatrix} \lambda & 0 \\ 0 & - \lambda
\end{pmatrix} \right).
\] $E_L$ has a natural lift $\hat{\sigma}: E_L \rightarrow E_L$ of $\sigma$ induced by the pullback. Choose a local frame $s$ of $L$ at a branch point $p \in Y$. Then 
\[ \hat{\sigma}=\begin{pmatrix} 0 & 1 \\ 1 & 0 \end{pmatrix}
\] with respect to the frame $s, \sigma^*s$ of $E_L$. Fix the diagonalizing frame 
\begin{align} s_+=s \oplus \sigma^*s , \qquad s_-=s \oplus - \sigma^*s. \label{equ:s+-}
\end{align}
\begin{lemm}\label{lemm:Hecke:inv} Let $\alpha\in H^0(\Lambda-D, E_L)^*$. Then $\hat{E}_L^{(\Lambda-D,\alpha)}$ is a $\sigma$-invariant holomorphic vector bundle if 
\begin{itemize}
\item[i)] for all $p \in \Fix(\sigma)$, $[\hat{\sigma}\alpha]_p=-[\alpha]_p$, and
\item[ii)] for all $p \in \supp(\Lambda-D)\setminus \Fix(\sigma)$, $[\hat{\sigma} \alpha]_p=[\alpha]_{\sigma p}$.
\end{itemize}
\end{lemm}
\begin{proof}
Let $p \in \supp(\Lambda-D)\setminus\Fix(\sigma)$. If ii) is satisfied, the lift $\hat{\sigma}$ induces a lift of $\sigma$ on $\hat{E}_L^{(\Lambda-D,\alpha)}$ at $p$ . Let $p \in \Fix(\sigma)$. From i) $\alpha$ can be extended to a local section $s_\alpha$ around $p$ such that $s_+,s_\alpha$ is a frame of $E_L$. With respect to this frame, the Hecke transformation is equivalent to introducing the new transition function 
\[ \psi_{01}= \begin{pmatrix} 1 &0 \\ 0 & z^{\Lambda_p-D_p} \end{pmatrix}
\] (with notation as in \ref{equ:trans:Hecke1},\ref{equ:trans:Hecke2}). Because $(\Lambda-D)_p \equiv 1 \mod 2$, we conclude that the induced frame of $\hat{E}_L^{(\Lambda-D,\alpha)}$ at $p$ is $\sigma$-invariant. In conclusion, $\hat{E}_L^{(\Lambda-D,\alpha)}$ is a $\sigma$-invariant holomorphic bundle (cf. Definition \ref{def:sigma-inv}).
\end{proof}

Let $D \in \Div^+(Y)$ be an effective divisor. Define 
\begin{align*} H^0(D,(E_L,\hat{\sigma}))&:= \{ \alpha \in H^0(D,E_L) \mid \alpha \text{ satisfies i) and ii) } \} \\
H^0(D,(E_L,\hat{\sigma}))^*&:= \{ \alpha \in H^0(D,E_L)^* \mid \alpha \text{ satisfies i) and ii) } \} \\
G_D&:=\{ \phi \in H^0(D,\O_Y)^* \mid \sigma^* \phi=\phi \}.
\end{align*}

\begin{prop}\label{prop:u-coord}
Let $\supp \Lambda=\Fix(\sigma)$, $D$ a $\sigma$-Higgs divisor and $N:=\Lambda-D \in \Div^+(Y)$. Let $(U,z)$ be a union of coordinate neighbourhoods $(U_p,z_p)$ around $p \in \Fix(\sigma)$ disconnecting $\Fix(\sigma)$, such that $\sigma: z_p \mapsto -z_p$. Let $s \in H^0(U,L)$ be a frame. Then there is a holomorphic surjective map
\[ u_s: H^0(N,(E_L,\hat{\sigma}))^* \rightarrow \bigoplus_{p \in \Fix(\sigma)} \C z_p + \C z_p^3+ \dots + \C z_p^{N_p - 2},
\] which factors through the action of $G_N$ and induces a bijection on the quotient. Let $s'=fs$ with $f=f_e+f_o \in H^0(U,\O_Y^*)$ and $f_e,f_o$ the even and odd part with respect to $\sigma$. Then
\[ u_{s'}=\frac{f_eu_s-f_o}{f_e-f_ou_s} \mod z^N.
\]
\end{prop}
\begin{proof}  
A choice of $s$ induces a frame $s_+=s+ \sigma^*s, \ s_-=s - \sigma^*s$ of $E_L$. With respect to such frame, we can explicitly parametrize 
\[ H^0(N,(E_L,\hat{\sigma}))^*=\{ [as_++bs_-] \mid a,b \in H^0(N,\O_Y): \sigma^*a=-a, \sigma^*b=b \}.
\] Let $N=lp$ with $l \in \N$, $p \in Y$. We define the isomorphism
\begin{align*} u_s: H^0(N,(E_L,\hat{\sigma}))^* &\rightarrow  \{u \in H^0(N,\O_Y) \mid \sigma^*u=-u\} \\
 [as_++bs_-] \qquad &\mapsto \qquad \tfrac{a}{b} \mod z^{N_p}.
 \end{align*} This map clearly factors through $G_N$ and separates the orbits. The right side is a $\C$ vector space over the basis $z,z^3,\dots, z^{N_p-2}$. Let now $s'=(f_e+f_o)s$, then 
 \[ s_+'=f_es_++f_os_-, \quad s_-'=f_es_-+f_os_+
 \] and
 \[ \alpha= \frac{1}{f_e^2-f_o^2}\left( (a f_e-bf_o)s_+'+(-af_o+f_eb)s_-') \right).
 \] Applying $u_{s'}$ gives the result.
\end{proof}

\begin{prop}\label{prop:Hecke_bundle} Let $\supp \Lambda=\Fix(\sigma)$, $D$ a $\sigma$-Higgs divisor and $N=\Lambda-D$. Then 
\[ F_D=\left\{ H^0(N,(E_L,\hat{\sigma}))^*/G_N \mid L \in \Pr_{N}(Y) \right\} \rightarrow   \Pr_{N}(Y)\]
is a holomorphic vector bundle of rank
\[ r=\sum\limits_{y \in Z(\lambda)} \tfrac12 (\Lambda_y-D_y-1)= \tfrac12 \left(\deg(\Lambda)-\deg(D)-\#\mathsf{Fix}(\sigma)\right).
\]
\end{prop}
\begin{proof} There exists a universal line bundle
\[ \mathcal{L} \rightarrow Y \times \Pr_{N}. 
\] Let $U$ be a disconnecting neighbourhood of $\Fix(\sigma)$ as above. A local trivialisation of $\mathcal{L}$ over $U \times V \subset Y \times \Pr_{N}$ is equivalent to choosing a local frame $s \in H^0(U,L)$ over $V \in \Pr_{N}$ in a coherent way. It defines a $u$-coordinate of $F_D$ over $V$, in other words, a local trivialisation $F_D \rest_V \cong V \times \C^r$. Changing the trivialisation corresponds to choosing a different holomorphic frame $s' \in H^0(U,L)$. The corresponding transformation of the $u$ coordinate is holomorphic by the previous proposition.
\end{proof}

\begin{theo}\label{theo:strata:Hecke-bundle:odd}
Let $\supp \Lambda= \Fix(\sigma)$ and $D$ a $\sigma$-Higgs divisor. Then there is an isomorphism of $F_D \rightarrow S_D$ making the following diagram commute:
\begin{center}
\begin{tikzpicture}  \matrix (m) [matrix of math nodes,row sep=3em,column sep=4em,minimum width=2em]
  {
     F_D & \mathcal{S}_D \\
    \mathsf{Prym}_{\Lambda-D} & \mathsf{Prym}_{\Lambda-D}  \\
  };
  \path[-stealth]
    (m-1-1) edge  (m-2-1)
          	edge (m-1-2)
    (m-1-2) edge node [left] {$\Eig_D$}(m-2-2)
    (m-2-1) edge node [above] {$\id$} (m-2-2);
\end{tikzpicture}
\end{center}
In particular,
\[ \dim \mathcal{S}_D = \deg(M)-\tfrac12\deg(D) + g(Y/\sigma)-1.
\]
\end{theo}
\begin{proof} Let $N=\Lambda-D$. Let $L \in \Pr_{N}(Y)$. Let $(U,z)$ be a union of coordinate neighbourhood of $\Fix(\sigma)$ disconnecting $\Fix(\sigma)$, such that $\sigma: z \mapsto -z$. Let $t \in H^0(U,p^*M)$ and $s \in H^0(U,L)$ local frames. We will show in i) how to produce Higgs bundles in $\S_D$ by applying Hecke transformation to $(E_L,\Phi_L)$. This defines the map $F_D \rightarrow \mathcal{S}_D$. To see that it is an isomorphism, we will show in ii) how to recover the $u$-coordinate from $(E,\Phi) \in \S_D$. 
\begin{itemize}
\item[i)] Let $\alpha \in H^0(N, (E_L,\hat{\sigma}))^*$, we saw in Lemma \ref{lemm:Hecke:inv} that $\hat{E}^{(N,\alpha)}$ with the induced lift $\hat{\sigma}$ is a $\sigma$-invariant holomorphic vector bundle. Furthermore, 
\[ \det(\hat{E}^{(N,\alpha)})=\det(E_L)(N)=L \otimes \sigma^*L \otimes \O(\Lambda-D)= \O_Y.
\] The Higgs field $\Phi_L$ induces a Higgs field on $\hat{E}^{(N,\alpha)}$. From Lemma \ref{lemm:trans:Hecke} it is easy to see that the Hecke transformation of $E_L$ in direction $\alpha=u s_++s_-$ is given by introducing a new transition function
\[ \hat{\psi}_{01}= \begin{pmatrix} 1 & -uz^{-N} \\ 0 & z^{-N} \end{pmatrix}
\] with respect to $s_+,s_-$ at every $p \in \supp \Lambda$ (with notation as in \ref{equ:trans:Hecke1},\ref{equ:trans:Hecke2}). In particular, the induced Higgs field is given by 
\begin{align} \hat{\Phi}_L^{(N,\alpha)}= \psi_{01}^{-1} \Phi_L \psi_{01}= \begin{pmatrix} -u z^{\Lambda} & z^D(1-u^2) \\
z^{2\Lambda-D} & u z^\Lambda \end{pmatrix} t.\label{equ:loc:form1}
\end{align} This is a well-defined Higgs field on $\hat{E}^{(N,\alpha)}$ with $\det(\hat{\Phi}_L^{(N,\alpha)})=-\lambda^2$ and vanishing divisor $D$. In conclusion, 
\[ (\hat{E}^{(N,\alpha)},\hat{\Phi}_L^{(N,\alpha)}) \in \Eig_D^{-1}(L) \subset \S_D.
\]
\item[ii)] Let $(E,\Phi) \in \Eig_D^{-1}(L) \subset \S_D$. Fix inclusions 
\[ i_+: L \rightarrow E, \quad i_-: \sigma^*L \rightarrow E
\] onto the corresponding subbundles. (There are $\C^*$-many such inclusions, but the coordinate will not depend on this choice.) Let 
\[ s_+:=i_+(s) + i_-( \sigma^*(s)), \quad s_-:=i_+(s) - i_-( \sigma^*(s))\in H^0(U,E).
\] As $\sigma$ is fixing the fibre over $y \in \Fix(\sigma)$, $s_+$ is non-vanishing. On the other hand, $s_-$ has a zero of odd order at $y$. We augment $s_+$ to a $\sigma^*$-invariant frame $s_+,s_a$ of $E$. With respect to this frame $s_-= -u s_+ + u' s_a$ with holomorphic odd functions $u,u' \in H^0(U,\O_Y)$. After acting on this frame by gauge transformations of the form
\[ \begin{pmatrix} 1 & \phi_1 \\ 0 & \phi_2 \end{pmatrix} 
\] with $\phi_i\in H^0(U,\O_Y)$, such that $\sigma^*\phi_i=\phi_i$, we may assume that $u'=z^n$ and $u$ is a polynomial of degree $<n$. The $\sigma$-invariant section $s_a$ is uniquely defined by these conditions. 
By definition $\Phi s_+=\lambda s_-$ and $\Phi s_-= \lambda s_+$. Hence, the Higgs field is given by
\begin{align} \Phi=\begin{pmatrix} -u z^{\Lambda_p} & z^{\Lambda_p-n} (1-u^2) \\ z^{\Lambda_p+n} & u z^{\Lambda_p} \end{pmatrix} t \label{equ:loc:form2}
\end{align} with respect to the invariant frame $s_+,s_a$. Computing the vanishing order of $\Phi$ we see that $n=\Lambda_y-D_y$ at $y \in \Fix(\sigma)$. Moreover, $u$ defines a polynomial germ $[u] \in H^0(N,\O_Y)$, such that $\sigma^*[u]=-[u]$. If we apply i) to $\alpha=[us_++s_-] \in H^0(N,(E_L,\hat{\sigma}))^*$ the Hecke transformation $(\hat{E}_L^{(N,\alpha)},\hat{\Phi}_L^{(N,\alpha)})$ is isomorphic to $(E,\Phi)$. There is a trivial isomorphism on $Y \setminus \Fix(\sigma)$, where the two Higgs bundles are isomorphic to $L \oplus \sigma^*L$. With respect to the fixed frame $s_+,s_a$ on $(E,\Phi)$ and the induced frame on $(\hat{E}_L^{(N,\alpha)},\hat{\Phi}_L^{(N,\alpha)})$ this extends to an isomorphism of holomorphic vector bundles. Local descriptions of the Higgs field with respect to these frames were computed in \ref{equ:loc:form1}, \ref{equ:loc:form2} and agree. Hence, the map extends to an isomorphism of Higgs bundles. In conclusion, the map $F_D \rightarrow \mathcal{S}_D$ defined in i) is a fibrewise isomorphism. 
\end{itemize}
\end{proof}

\begin{coro} Consider $\M^\sigma(Y,p^*M,\lambda)$, such that $Z(\lambda)=\mathsf{Fix}(\sigma)$ and $\lambda$ has only simple zeros. Then 
\[ \M^\sigma(Y,p^*M,\lambda)=\Pr_\Lambda. \]
\end{coro}
\begin{proof}
In this case, the only $\sigma$-Higgs divisor is $D=0$. So by Theorem \ref{theo:strat:general} and Theorem \ref{theo:strata:Hecke-bundle:odd}, we have $\M^\sigma(Y,p^*M,\lambda)\cong \mathcal{S}_0 \cong F_0 \cong \Pr_\Lambda$. 
\end{proof}

To describe the extra data for zeros of even order, it is more convenient to use extension classes. We will give an interpretation in terms of Hecke parameters in Proposition \ref{prop:Hecke:even}.

\begin{prop}\label{prop:even:zeros}
Consider $\M^\sigma(Y,p^*M,\lambda)$, such that $Z(\lambda)=\mathsf{Fix}(\sigma)\sqcup \{y,\sigma^*y\}$, the zeros $y$ and $\sigma^*y$ are of order $m \geq 1$, and all other zeros are simple. Fix a $\sigma$-Higgs divisor $D$ and $L \in \Pr_{\Lambda-D}$. Then 
\begin{align*} \mathsf{Eig}_D^{-1}(L) &\cong \{ [c] \in H^0(\Lambda_y y, L^2 p^*M) \mid \ord_y[c]=D_y\} \cong \C^* \times \C^{\Lambda_y-D_y-1}. 
\end{align*} The last isomorphism is determined by the choice of a local coordinate $(U,z)$ centred at $y$ and the choice of a local frame $s \in H^0(U,L^2p^*M)$.
\end{prop}
\begin{proof} By assumption, we can trivialize the covering in a neighbourhood $U$ of $p(y) \in Y/\sigma$. Let $p^{-1}(U)=U_y \sqcup U_{\sigma y}$, such that $y \in U_y$. Let $(E,\Phi) \in  \mathsf{Eig}_D^{-1}(L)$. By the $\sigma$-invariance $(E,\Phi)\rest_{p^{-1}(U)}$ is uniquely determined by $(E,\Phi)\rest_{U_y}$. So we need to parametrise the possible $(E,\Phi)\rest_{U_y}$ with eigenvalues $\pm \lambda\rest_{U_y}$. Regarding $(E,\Phi)\rest_{U_y}$ as a $\SL(2,\C)$-Higgs bundle with reducible spectral curve, we will use the description of $\SL(2,\C)$-Hitchin fibres with this property given in \cite{GO13} Section 7. Write $(E,\Phi)$ as an extension
\[ 0 \rightarrow (L,\lambda) \rightarrow (E,\Phi) \rightarrow (L^*,-\lambda) \rightarrow 0.
\] These extensions are parametrised by the hypercohomology group 
\[ \mathbb{H}^1(Y,(L^2,2\lambda))
\] of the complex of locally free sheaves
\[ \O_Y(L^2) \xrightarrow{2\lambda} \O_Y(L^2 p^*M).
\] The 5-term exact sequence of (one of) the associated spectral sequences reveals that 
\[ \mathbb{H}^1(Y,(L^2,2\lambda)) \cong H^0( \Lambda, L^2 p^*M)=\bigoplus\limits_{y \in \ \mathsf{supp}(\Lambda)}\O(L^2 p^*M)_y / \sim,
\] where for $v,v' \in \bigoplus\limits_{y \in \ \mathsf{supp}(\Lambda)}\O(L^2 p^*M)_y$
\[ v \sim v' \quad \Leftrightarrow \quad v=v'+f\lambda \qquad \mathrm{with} \quad f \in \bigoplus\limits_{y \in \mathsf{supp}(\Lambda)}\O_y.
\] By Theorem \ref{theo:strata:Hecke-bundle:odd}, the extension data at the simple zeros in $\Fix(\sigma)$ is uniquely determined by $\sigma$-invariance. Hence, the fibres of $\mathsf{Eig}_D$ are parametrised by 
\[ H^0(\Lambda_y y, L^2 p^*M),
\] where we consider $\Lambda_y y$ as a divisor supported at the point $y$. Furthermore, one can explicitly construct a Higgs bundle 
\[ \left( E=L \oplus_{C^\infty} L^{-1}, \delb_E= \begin{pmatrix} \delb_L & b \\ 0 & \delb_{L^{-1}} \end{pmatrix}, \Phi= \begin{pmatrix} \lambda & c \\ 0 & -\lambda \end{pmatrix}\right) 
\] from the extension data $[c] \in H^0( \Lambda, L^2 p^*M)$ by extending $[c]$ to a smooth section $c \in \A^{0}(Y,L^2 p^*M)$ and solving the equation 
\[ \delb c = 2b \lambda.
\] for $b \in \A^{(0,1)}(Y, L^2)$. In this way, we see that $\div(\Phi)_y=D_y$ if and only if $D_y= \ord_y c$. So $(E,\Phi) \in  \mathsf{Eig}_D^{-1}(L)$ are parametrized by the polynomial germs $[c] \in H^0(\Lambda_y y, L^2 p^*M)$ with $\ord_y([c])=D_y$. With respect to a coordinate $(U,z)$ centred at $y$ and a frame $s$ of $L^2 p^*M\rest_U$, such polynomial germ $[c]$ has a unique representative of the form
\[ \left(c_{D_y}z^{D_y}+c_{D_y+1} z^{D_y+1} + \dots + c_{\Lambda_y-1}z^{\Lambda_y-1}\right)s.
\] with $c_{D_y} \in \C^*$ and $c_{D_y+1}, \dots, c_{\Lambda_y-1} \in \C$. This defines the non-canonical isomorphism to $ \C^* \times \C^{\Lambda_y-D_y-1}$.
\end{proof}

\begin{theo}\label{theo:strat:general}
Fix a moduli space  $\M^\sigma(Y,p^*M,\lambda)$ and a compatible $\sigma$-Higgs divisor $D$. Then the stratum indexed by $D$ is a holomorphic fibre bundle
\[ (\C^\times)^{r_1} \times \C^{r_2} \rightarrow \mathcal{S}_D \rightarrow \Pr_{\Lambda-D}
\] with 
\begin{align*} r_1&=\tfrac12\left( \# Z(\lambda) -\# \Fix(\sigma)\right) \\
r_2&= \tfrac12 \left(\deg(\Lambda)-\deg(D) - \# Z(\lambda)\right).
\end{align*}
In particular, the dimension of the stratum is given by
\[ \deg(M)-\tfrac12 \deg(D) +   g(Y/\sigma) -1.
\]
\end{theo}
\begin{proof}
All the extension data depends only on the structure of the Higgs bundle at $Z(\lambda)$. So if we have more than one higher order zero in $\lambda$ the fibre of $\Eig_D$ is a Cartesian product of the fibres described in Theorem~\ref{theo:strata:Hecke-bundle:odd} and Proposition \ref{prop:even:zeros}. The coordinates in Proposition \ref{prop:even:zeros} depend holomorphically on the choice of a local frame $s \in H^0(U,L^2 p^*M)$. Hence, the argument given in the proof of Theorem \ref{theo:strat:general} establishes the structure of a holomorphic fibre bundle on $\mathcal{S}_D$.
\end{proof}

\section{Stratification of Singular Hitchin Fibres}\label{sec:stra:sl2c}
In this section, we specialize the results of the previous section to the singular Hitchin fibres in the moduli space of $K$-twisted Higgs bundles on $X$.

\begin{defi}
Let $q_2 \in H^0(X,K^2)$. A Higgs divisor associated to $q_2$ is a divisor $D \in \Div(X)$ supported at $Z(q_2)$, such that for all $p \in Z(q_2)$
\[ 0 \leq D_p \leq \lfloor \frac12 \ord_p(q_2) \rfloor, 
\] where $\lfloor \cdot \rfloor$ denotes the floor function.
\end{defi}

For $q_2 \in H^0(X,K^2)$ let
\begin{align*} n_{\even}=\# \{ p \in Z(q_2) \mid p \text{ zero of even order} \} \\
n_{\odd}=\# \{ p \in Z(q_2) \mid p \text{ zero of odd order}. \}
\end{align*}

\begin{theo}\label{theo:strat:sl2C}
Let $q_2 \in H^0(X,K^2)$ be a quadratic differential with no global square root on $X$. Then there exists a stratification
\[ \H^{-1}(q_2)= \bigcup\limits_{D} \mathcal{S}_D
\] by locally closed sets $\mathcal{S}_D$ indexed by Higgs divisors $D$. If $n_{\odd} \geq 1$, each stratum has the structure of a holomorphic fibre bundle 
\[ (\C^\times)^{r_1} \times \C^{r_2} \rightarrow \mathcal{S}_D \rightarrow \Pr_{\Lambda-D}(\tilde{\Sigma}),
\] where
\begin{align*} r_1&=n_{\mathrm{even}}, \quad \text{and} \quad
r_2= 2g-2- \deg(D)-n_{\even}-\frac12n_{\odd}.
\end{align*}  If $n_{\odd}=0$, $\mathcal{S}_D$ is a branched two-to-one covering of a holomorphic fibre bundle over 
\[ \Nm^{-1}(I(D-\tfrac12\div(q_2)))
\] with fibres given as in the previous case. Here $I$ denotes the unique non-trivial line bundle $I \in \ker(\tpst)$. In general, the dimension of a stratum $\mathcal{S}_D$ is given by
\[ 3g-3-\deg(D).
\]  
\end{theo}
\begin{proof} The stratification by Higgs divisors is obtained in the same way as in Theorem \ref{theo:stratifi}. We analysed the map $\tpst: \H^{-1}(q_2) \rightarrow \M^\sigma(\tilde{\Sigma},\tpst K,\lambda)$ in Section \ref{sec:pap}. It is bijective, if there is at least one zero of odd order, and generically two-to-one, if $n_{\odd}=0$. We showed above how $\M_\lambda^\sigma$ is stratified by $\sigma$-Higgs divisors (Theorem \ref{theo:stratifi}). Here a $\sigma$-Higgs divisors on $(\tilde{\Sigma},\sigma,\lambda)$ is the pullback $\tpst D$ of a Higgs divisor $D \in \Div(X)$ associated to $q_2$. The structure of the strata was described in Theorem \ref{theo:strat:general}. We have $\# \mathsf{Fix}(\sigma)=n_\odd$, $\# Z(\lambda)=n_\odd+ 2n_\even$, and 
\[ g(\tilde{\Sigma})=2g-1 + \tfrac{n_\odd}{2}.
\] Hence, the dimension of the stratum $\mathcal{S}_D$ is given by
\[ \deg(K)-\tfrac12 \deg(\tpst D) +   g -1=3g-3-\deg(D).
\] 

We are left with showing that for $n_{\odd}=0$ the eigenline bundles $\O(L)=\Ker(\tpst\Phi -\lambda \id_E)$ lies in 
\[ \Nm^{-1}(I(D-\tfrac12\div(q_2))) \subsetneq \Pr_{\Lambda-\tpst D}.
\] In this case, the $\sigma$-invariant pushforward of a $\sigma$-invariant vector bundle $(E,\hat{\sigma})$ with $\det(E)=\O_{\tilde{\Sigma}}$ has determinant $\O_X$ or $I$ depending on the lift $\hat{\sigma}$. If $(E,\Phi)$ is in the lowest dimensional stratum, i.e.\ $D=\frac12 \div(q_2)$, then 
\begin{align*} \det(E)&=\det(\tilde{\pi}_*^\sigma(L \oplus \sigma^*L, \hat{\sigma}))= \det(\tilde{\pi}_*L) \\ 
&= \Nm(L) \otimes \det(\tilde{\pi}_*\O_{\tilde{\Sigma}})=\Nm(L) \otimes I. 
\end{align*} using Lemma \ref{lem:2push} and Corollary \ref{coro:push:trivial}. Hence, $\Nm(L)=I$. Similarly, for the higher strata $\det(E)=\O_X$ if and only if $\Nm(L)=I(D-\frac12\div(q_2))$. 
\end{proof}

\subsection{Algebraic and geometric interpretation}

The torus structure is not completely lost, when we degenerate to the singular locus.  Starting with a torus of codimension 1 in the first degeneration, the deeper we go into the singular locus the lower the dimension of the subtorus becomes. All singular fibres with irreducible and reduced spectral curve have a subtorus of at least dimension $g-1$. Furthermore, these subtori are torsors over the Prym variety of the normalised spectral cover. \\
The highest dimensional stratum $\mathcal{S}_0$ corresponds to Higgs bundles $(E,\Phi) \in \H^{-1}(q_2)$ with non-vanishing Higgs field. One can show, that this stratum corresponds to the locally free sheaves on the singular spectral curve by the Beauville-Narasimhan-Ramanan correspondence \cite{BNR89}. The lowest dimensional stratum $\mathcal{S}_{D_\text{max}}$ contains the Higgs bundles $(E,\Phi) \in \H^{-1}(q_2)$ with maximal vanishing order. At a zero of odd order $2m+1$, they can be locally written as 
\[ \Phi(z)=z^m\begin{pmatrix} 0 & 1 \\ z & 0 \end{pmatrix} \d z
\] and at a zero of order $2m$, they are diagonalizable with local eigensections $\pm z^m \d z$. If $B \in \Div(X)$ is the branch divisor of the covering by the normalised spectral curve, we have
\[ D_{\text{max}}= \Lambda- \tfrac12 \tpst B.
\]
This stratum $\mathcal{S}_{D_\text{max}}$ has no Hecke parameters or extension data. It is obtained as the algebraic pushforward along $\tilde{\pi}: \tilde{\Sigma} \rightarrow X$ of the line bundles $L(\tfrac12 \tpst B)$ for $L$ in the torsor over the Prym variety
\[ \Pr_{\tfrac12 \tpst B}(\tilde{\Sigma}) \quad \text{resp.} \quad \Nm^{-1}(I),
\] if $\tilde{\pi}$ is unbranched. Restricted to certain sets of quadratic differentials in the singular locus one recovers the subintegrable systems described by Hitchin \cite{Hi19}.

\section{Singular Fibres with Locally Irreducible Spectral Curve}\label{sec:glob:odd}
In this section, we start analysing how the strata glue together to form the singular Hitchin fibre. We will consider the case, where the spectral curve is locally irreducible, i.e.\ the quadratic differential has only zeros of odd order. To do so, we need to compactify the moduli of Hecke parameters of the highest stratum. We show that these singular Hitchin fibres are themselves holomorphic fibre bundles over twisted Prym varieties with fibres given by the compactified moduli of Hecke parameters. This allows an explicit description of the fibres for the first degenerations. \\

\subsection{Hecke transformations - revisited}
We will again work in the setting introduced in Section \ref{sec:strat}. So $(Y,\sigma)$ is Riemann surface with an involution $\sigma$, $p: Y \rightarrow Y/\sigma$ the associated two-covering, $M$ a line bundle on $Y/\sigma$ and $\lambda: Y \rightarrow p^*M$ a section with the properties described there. In this section, we consider $\M_\lambda^\sigma=\M^\sigma(Y,p^*M,\lambda)$, such that $Z(\lambda)=\Fix(\sigma)$. Under identifying $\M_\lambda^\sigma$ with the Hitchin fibre via  pullback, this extra condition is equivalent to the spectral curve being locally irreducible. 

Let us shortly recall how we used Hecke transformations in Section \ref{sec:strat}. We saw in Theorem \ref{theo:strata:Hecke-bundle:odd} that for fixed $L \in \Prym_{\Lambda}(Y)$ the Higgs bundles $(E,\Phi) \in \mathcal{S}_0$, which project to $L$, are parametrized by 
\[ \Eig_\Lambda^{-1}(L)= H^0(\Lambda,(E,\hat{\sigma}))^* / G_\Lambda.
\] After choosing frames of $L$ at $Z(\lambda)$, the Hecke parameters are encoded in the polynomial germs $u$. We reconstructed a $\sigma$-invariant Higgs bundle from the spectral data $(L,u)$ as the Hecke transformation of $(E_L,\Phi_L)$ at $\Lambda$ in direction of $\alpha=us_++s_-$ introducing the new transition function 
\[ \hat{\psi}_{01}= \begin{pmatrix} 1 & -uz^{-\Lambda_p} \\ 0 & z^{-\Lambda_p} \end{pmatrix}
\] 
(see Theorem \ref{theo:strata:Hecke-bundle:odd}). To compactify the moduli of Hecke parameters, we need to allow Hecke parameters $\alpha \in H^0(\Lambda,(E,\hat{\sigma}))$, which may vanish on $\supp \Lambda$. Fix $L \in \Prym_{\Lambda}(Y)$ and a frame $s \in H^0(U,L)$ in a neighbourhood $U$ of $Z(\lambda)$.

\begin{defi}\label{def:higher:Hecke} Let $\alpha \in H^0(\Lambda,(E,\hat{\sigma}))\setminus\{0\}$. Define $(\hat{E}_L^{(\Lambda,\alpha)},\hat{\Phi}_L^{(\Lambda,\alpha)})$ by introducing a new transition function 
\[ \hat{\psi}_{01}^y= \begin{pmatrix} b^{-1} & -az^{-\Lambda_p} \\ 0 & bz^{-\Lambda_p}
\end{pmatrix},
\] with respect to $s_+,s_-$ for all $y \in Z(\lambda)$, where $a,b$ are defined through $\alpha=as_++bs_-$ (see Section \ref{sec:Hecke} for details on the notation).
\end{defi}

Recall that the Hecke transformation is invariant under the group action of 
\[ G_\Lambda= \{ \phi \in H^0(\Lambda,\O_Y)^* \mid \sigma^* \phi = \phi \}
\] on the Hecke parameters $H^0(\Lambda,(E,\hat{\sigma}))^*$. When we allow the Hecke parameters to vanish, there is another equivalence relation.

\begin{lemm}\label{lemm:equ:ii}
\begin{itemize}
\item[i)] Let $\alpha \in H^0(\Lambda,(E_L,\hat{\sigma}))$ and $\phi \in G_\Lambda$. Then $\hat{E}_L^{(\Lambda,\alpha)} \cong \hat{E}_L^{(\Lambda,\phi\alpha)}$. 
In particular, Definition \ref{def:higher:Hecke} and Definition \ref{defi:dual:Hecke} agree for $\alpha \in H^0(\Lambda,(E_L,\hat{\sigma}))^*$.
\item[ii)] Let $\alpha,\alpha' \in H^0(\Lambda,(E_L,\hat{\sigma}))$, such that $\div(\alpha)=\div(\alpha')=D$. Then $\hat{E}_L^{(\Lambda,\alpha)} \cong \hat{E}_L^{(\Lambda,\alpha')}$, if the projections of $\alpha,\alpha'$ to $H^0(\Lambda-D,(E_L,\hat{\sigma}))$ agree. 
\end{itemize}
\end{lemm}

\begin{proof}
In i), the new transition function of the Hecke transformation in direction of $\phi\alpha$ is given by 
\[ \begin{pmatrix} \phi^{-1} b^{-1} & -\phi az^{-\Lambda_p} \\ 0 & \phi bz^{-\Lambda_p} \end{pmatrix}= \begin{pmatrix} b^{-1} & -az^{-\Lambda_p} \\ 0 & bz^{-\Lambda_p} \end{pmatrix} \begin{pmatrix} \phi^{-1} & 0 \\ 0 & \phi \end{pmatrix}
 \] Hence, the transition functions define isomorphic vector bundles. For ii), let $\phi \in H^0(\Lambda,\O_Y)$, such that $\sigma^*\phi=\phi$. Then the Hecke transformations with respect to $\alpha=as_+ + bs_-$ and $\alpha'=(a + \phi \frac{\lambda}{b})s_++bs_-$ are isomorphic. The isomorphism is given by the gauge transformation
\[ \begin{pmatrix}b^{-1} & -az^{-\Lambda_p} \\ 0 & bz^{-\Lambda_p} \end{pmatrix}\begin{pmatrix} 1 & \phi \\ 0 & 1 \end{pmatrix} = \begin{pmatrix} b^{-1} & (\frac{\phi}{b}z^{\Lambda_p}-a)z^{-\Lambda_p} \\ 0 & bz^{-\Lambda_p}
\end{pmatrix}.
\] This provides the equivalence in the $a$-coordinate. Using the $G_\Lambda$ action one obtains equivalence ii).
\end{proof}

\subsection{Weighted projective spaces}\label{ssect:weighted:proj}
We will obtain a topological model for the compact moduli of Hecke parameters by gluing subsets of weighted projective spaces. Let us recall some basic facts about weighted projective spaces. \\

A weight vector $(i_0, \dots, i_n) \in \N^n$ defines a $\C^*$-action on $\C^{n+1}$ by
\[ \C^* \times \C^{n+1} \rightarrow \C^{n+1}, \quad (\lambda,x_0,\dots,x_n) \mapsto (\lambda^{i_0}x_0, \dots, \lambda^{i_n}x_n).
\] The weighted projective space $\P(i_0,\dots,i_n)$ is defined as the quotient of $\C^{n+1}\setminus\{(0,\dots,0)\}$ by this action. We will denote the equivalence class of $(x_0,\dots,x_n)$ by $(x_0:\dots:x_n)$. 

Weighted projective spaces are complex orbifolds. 
We obtain orbifold charts in the same way one defines affine charts of projective space $\P^n$. For example, for points of the form $(1,x_1,\dots,x_n) \in \C^{n+1}$, the $\C^*$-action restricts to an action of $\Z_{i_0}$ given by 
\[ (1,x_1,\dots,x_n) \mapsto (1,\xi_{i_0}^{i_1}x_1,\dots,\xi_{i_0}^{i_n}x_n),
\] where $\xi_{i_0}$ is a primitive $i_0$-th root of unity. Weighted projective spaces are normal toric complex spaces. In an orbifold chart, the torus action is given by
\[ (1:x_1:\dots:x_n) \mapsto (1:\lambda_1^{i_1}x_1: \dots : \lambda_n^{i_n}x_n)
\] for $(\lambda_1, \dots, \lambda_n) \in (\C^*)^n$. This extends to an analytic action on $\P(i_0,\dots,i_n)$. We call a analytic subspace $Y \subset \P(i_1,\dots,i_n)$ toric, if it is preserved by the torus action.

\subsection{Compact moduli of Hecke parameters}\label{ssec:mod:Heck}
In this section, we want to study the compact moduli of Hecke parameters. To do so we will restrict our attention to the Hecke parameters at a single higher order zero. 

Let $(U,z)$ be holomorphic disc centred at $0 \in \C$ and $\sigma: z \mapsto -z$. Let $d \in \N$ a odd number and $D$ the divisor with coefficient $d$ at zero. Define 
\[ \mathsf{Heck}_d= \left. \left\{ \begin{pmatrix} a \\ b \end{pmatrix} \in H^0(D,\O_U^2) \ \middle| \ \sigma^*a=-a, \ \sigma^*b=b \right\} \middle/ \sim \right. ,	
\] where 
\begin{align*} \alpha=\begin{pmatrix}a \\ b \end{pmatrix} \sim \alpha'=\begin{pmatrix}a' \\ b' \end{pmatrix} \quad \Leftrightarrow  \quad \begin{array}{l} \ord_0(\alpha)=\ord_0(\alpha')=:n \\
\text{and }  \alpha=\alpha' \text{ mod } z^{d-n}.
\end{array}
\end{align*} These are the equivalence classes of relation ii) in Lemma \ref{lemm:equ:ii}. For $0 \leq {n} < \frac{d}{2}$ let 
\[ V_{n}:= \{ \alpha \in \mathsf{Heck}_d \mid  \ord_0(\alpha)=n \}.
\] 
We can understand the quotient of $\Heck_d$ by $G_D$ by gluing subsets, on which we find explicit invariant polynomials. By Proposition \ref{prop:G}, $G_D=\C^* \times H_D$, where
\[ H_D= \{ 1+\phi_2z^2+ \dots \in H^0(D, \O_U) \} 
\] is the maximal unipotent normal subgroup. We will first factor through $H_D$ as orbits of unipotent group actions on affine spaces are closed. The resulting intermediate quotient can be factored through $\C^*$. The subsets $V_n$ will correspond to the strata of the stratification \ref{theo:stratifi}.

\begin{lemm}\label{lemm:u-coord:2} Let $0 \leq n \leq \frac{d-3}{2}$. There is a holomorphic map $u_n: V_n \rightarrow \P^{\frac12(d-2n-1)}$ invariant under the $G_D$-action and separating the orbits. Its image is an affine chart of $\P^{\frac12(d-2n-1)}$. For $n=\frac{d-1}{2}$, the $G_D$ action identifies $V_n$ to a point.
\end{lemm}
\begin{proof} 
Let us assume $n \leq \frac{d-3}{2}$ is even. Every $\alpha \in V_n$ has a unique representative of the form 
\[ \begin{pmatrix} \frac{a_{n+1}z^{n+1}+\dots+a_{d-2}z^{d-2}}{1+\frac{b_{n+2}}{b_n}z^2+ \dots + \frac{b_{d-1}}{b_n}z^{d-1}} 	\\ b_n z^n
\end{pmatrix} \mod z^{d-n}
\] with $b_n \neq 0$ in its $H_D$-orbit. In particular, $b_n$ and the $(n+1)$-th, \dots, $(d-n-2)$-th derivatives of the fraction in the first coordinate define $\tfrac12(d-2n+1)$ holomorphic functions invariant under the $H_D$-action. This defines a map
\[ V_n \rightarrow \C^* \times \C^{\frac12(d-2n-1)}.
\]The $\C^*$-action acts with weight $1$ on every coordinate. By factoring through $\C^*$, we obtain the desired map to an affine chart of $\P^{\frac12(d-2n-1)}$.
For $n$ odd, every $\alpha \in V_n$ has a unique representative 
\[ \begin{pmatrix} a_{n}z^n \\ \frac{b}{1+\frac{a_{n-2}}{a_n}z^2+\dots + \frac{a_{d-2}}{a_2}z^{d-2}}
\end{pmatrix}  \mod z^{d-n} . 
\] By recording $a_n$ and the $(n+1)$-th, \dots, $(d-n-2)$-th derivative of the second coordinate and again factoring through the $\C^*$-action we obtain an invariant map 
\[ V_n \rightarrow \P^{\frac12(d-2n+1)}.
\] As $a_n \neq 0$, the image is an affine chart. If $n=\frac{d-1}{2}$ is odd, the only $H_D$-invariant function on $V_n$ is $a_n \neq 0$. Hence, $V_n$ is identified to a point by the $\C^*$-action. Similarly, for $n=\frac{d-1}{2}$ even.
\end{proof}

It seems impossible to find enough invariant functions to define the global quotient $\Heck_d/G_D$. However, we obtain a topological model by gluing the quotients of subsets, which are easier to understand.

\begin{prop}\label{prop:N_l^n} There exist finitely many locally closed connected subsets $N_i \subset \mathsf{Heck}_d$, $i \in I$, such that 
\begin{itemize}
\item[i)] for every $n < l \leq \frac{d-1}{2}$ and $\alpha \in V_l$, there exist $i \in I$, such that $\alpha \in N_i$ and $N_i \cap V_n \neq 0$,
\item[ii)] there exist algebraic maps $N_i \rightarrow \P(1,1, 2,3, \dots, m_{i})$ invariant by the action of $G_D$, which separate the $G_D$-orbits. Their images are toric subspaces and contain no singular points. 
\end{itemize}
\end{prop}
\begin{proof} For $n \leq l \leq \frac{d-1}{2}$ let 
\begin{align*} {N}_l^n:=H_D \cdot \{ &a= x_1z+ x_3z^3 + \dots , b=x_0z + x_2 z^2 + \dots \mod z^{d-n} \mid  \\ 
&x_0=  \dots= x_{n-1}= x_{n+1}=\dots =x_{l-1}=0, x_l \neq 0 \}.
\end{align*}
Let $\alpha=(a,b) \in N_l^n$. If $x_n \neq 0$, we have $\ord_0(\alpha)=n$, hence $\alpha \in V_n$. If $x_n=0$, we have $\alpha \in V_l$. So $N_l^n$ describes a locally closed subset of $V_n$ containing $V_l$ in its closure. 
We first want to find invariant polynomials by the $H_D$-action and then take the quotient by $\C^*$. Let $l$ be odd and $n$ be even, then
\begin{align*} a&=x_lz^l + \dots + x_{d-n-2}z^{d-n-2}, \\ b&= x_nz^n +x_{l+1}z^{l+1}+ \dots + x_{d-n-1}z^{d-n-1}
\end{align*} Every orbit in $N_l^n$ has a representative of the form
\begin{align} \begin{pmatrix} x_lz^{l} \\ \frac{x_nz^n+x_{l+1}z^{l+1}+ \dots + x_{d-n-1}z^{d-n-1}}{1+ \tfrac{x_{l+2}}{x_l}z^2+ \dots +\tfrac{x_{d-n-2}}{x_l}z^{d-n-2-l}}
\end{pmatrix} \mod z^{d-n}. \label{equ:4.1}
\end{align} The $n$-th, $n+2$-th, \dots, $(d-l-2)$-th derivative give (after multiplying with their common divisor) 
\[ \tfrac12 (d-l-n)
\] homogeneous polynomials of degree 
\[ 1, 2, \dots, \tfrac12 (d-l-n).
\] The representative in \ref{equ:4.1} is not quite unique because if we act by $(1+z^{d-n-l}\phi) \in H_D$ the $a$-coordinate stays unchanged modulo $z^{d-n}$. However, as we only record up to the $d-l-2$-th derivative of the $b$-coordinate these homogeneous polynomials are invariant under the $H_D$-action. Furthermore, it is easy to see that they are independent elements of the algebra of $H_D$-invariant polynomials on $N_l^n$ because the $n+2k$-th derivative is the first one to contain $x_{l+2k}$. By recording $x_l$ in addition, we have $\tfrac12 (d-l-n+2)$ independent homogeneous polynomials. This defines a map
\[ N_l^n \rightarrow \C^{\tfrac12 (d-l-n+2)}
\] invariant by the $H_D$-action. Factoring through the $\C^*$-action we obtain the desired algebraic map
\[ N_l^n \rightarrow \P(1,1,2, 3, \dots, \frac{d-l-n}{2}).
\]
To show that it separates orbits we first consider $N_l^n \cap V_l$, i.e.\ $x_n = 0$. Here every element has a unique representative $\frac{b}{a} \mod z^{d-l}$. Those are determined by the invariant polynomials induced from the derivatives of order $l+1$ up to $d-l-2$. Instead on $N_l^n \cap V_n$ we can uniquely represent each element by a $u$-coordinate, see Lemma \ref{lemm:u-coord:2}. This $u$-coordinate can be recovered from the invariant polynomials. The $(n+2)$-nd derivative encodes $x_{l+2}$, the $(n+4)$-th $x_{l+4}$ etc. and the $(d-l-2)$-th encodes $x_{d-n-2}$. So the map separates orbits. \\
As $x_l \neq 0$ we see that the image is contained in
\[ \{(y_0: \dots : y_{\tfrac12 (d-l-n)}) \in \P(1,1,\dots, \tfrac12 (d-l-n)) \mid y_0 \neq 0 \}.
\] This subset contains no singularity of the weighted projective space. Furthermore, by the explicit description of the homogeneous polynomials it is easy to verify that the image is closed under the torus action of $(\C^*)^{\tfrac12 (d-l-n)}$. \\
Now, let us consider the case of even $n$ and even $l$. Here we have to take a finer decomposition. Let $k > l$ a odd number then
\begin{align*} \prescript{k}{}{N}_n^l:=H_D \cdot &\left\{  \begin{array}{l} a= x_1z+ x_3z^3 + \dots + x_{d-n-2}z^{d-n-2} \\ b=x_0z + x_2 z^2 +\dots + x_{d-n-1}z^{d-n-1} \end{array} \middle\vert \right.  \\ 
 & \left. \begin{array}{l} x_0=  \dots= x_{n-1}= x_{n+1}=\dots =x_{l-1}=0, x_l \neq 0, \\
x_{l+1}=x_{l+3}=\dots= x_{k-2}=0, x_k \neq 0 \end{array} \right\}.
\end{align*} 
Clearly
\[ \bigcup\limits_{\text{odd } k \geq l} \prescript{k}{}{N}_n^l=N_n^l.
\] So with these subsets we still satisfy property i). 
For fixed $k$ we proceed as before by computing the derivatives of order $n$ up to $(d-k-2)$ of
\[ \frac{x_nz^n+x_{l}z^{l}+ \dots + x_{d-1}z^{d-1}}{1+ \tfrac{x_{k+2}}{x_k}z^2+ \dots +\tfrac{x_{d-2}}{x_k}z^{d-2-l}}.
\] They define a map 
\[ \prescript{k}{}{N}_l^n \rightarrow \P(1,1,2,3 \dots, \tfrac 12 (d-k-n)).
\] invariant by the action of $G_D$. For $x_n \neq 0$ we can recover the $u$-coordinate of $\prescript{k}{}{N}_l^n \cap V_n$ as above. If $x_n = 0$ the $u$-coordinate of the lower stratum is now given by $\frac{a}{b} \mod z^{d-l}$. We recover $a_{k+2}$ from the $l+2$-th derivative, $a_{k+4}$ from the $l+4$-derivative till $a_{d-l-2}$ from the $d-k-2$-th derivative. These uniquely defines the $u$-coordinate on $V_l$. With the same argument as above the image contains no singular points and is closed under the torus action. \\
When $n$ is odd we can obtain the same results by changing the role of $a$ and $b$.
\end{proof}

\begin{theo}\label{theo:moduli_Hecke}
The quotient of $\Heck_d$ by the action of $G_D$ is a union of toric subspaces of weighted projective spaces glued algebraically along torus orbits.
\end{theo}
\begin{proof}
Most of the work was already done in the previous lemma by introducing the sets $N_i \subset \Heck_d$ and the $G_D$-invariant, orbit-separating maps 
\[ N_i \rightarrow \P(1,1,2,3, \dots ,m_i).
\] These maps identify the quotients $N_i/G_D$ with toric subspaces of weighted projective spaces. 
We can build a model for the quotient $\Heck_d/G_D$ by gluing together this subsets $N_i/G_D$ along their intersection. It is left to show that this happens algebraically along torus orbits. 
It is enough to show that for all $i \in I$ and $0 \leq l \leq \frac{d-3}{2}$ the intersection $N_i \cap V_l$ is mapped onto a toric subspace under the two maps to weighted projective spaces and that the coordinate change is polynomial. \\
We will show this for $N_l^n\cap V_n$ with $n < l \leq \frac{d-3}{2}$, $n$ even and $l$ odd. For the other cases, it works in the same way. Denote by 
\begin{align*}  u_n: V_n  \rightarrow \P^{\tfrac12(d-2n-1)}, \qquad f_l: N_l^n \rightarrow \P(1,1,2,\dots,\frac{d-l-n}{2})
\end{align*} the $G_D$-invariant maps defined in Lemma \ref{lemm:u-coord:2} and Proposition \ref{prop:N_l^n}. Let $\alpha \in N_l^n\cap V_n$. We can choose a representative of the form
\[ \begin{pmatrix} x_lz^l+x_{l+2}z^{l+2}+ \dots + x_{d-n-2}z^{d-n-2}  \\ z^n
\end{pmatrix}.
\] The image under $u_n$ is given by 
\[ (1:0:\dots:0: x_l:x_{l+2}:x_{d-n-2}) \in \P^{\tfrac12(d-2n-1)}.
\] So $u_n(N_l^n\cap V_n)$ is clearly a union of torus orbits. On the other hand, we can explicitly compute the values of the invariant polynomials defining $f_l$ and obtain
\begin{align*} &\left(  x_l \ : \ 1 \ : \ x_{l+2}\ :\ x_{l+4}x_l+x_{l+2}^2\ :\ \dots\ :\ x_{d-n-2}(x_l)^{\tfrac12 (d-l-2-n)}+ \dots \right) \\
&\in \P(1,1,2,\dots,\frac{d-l-n}{2}) .
\end{align*} This is again a union of torus orbits. It is easy to check that the gluing maps are polynomial.
\end{proof}

\begin{figure}
\begin{center}
\includegraphics[scale=0.45]{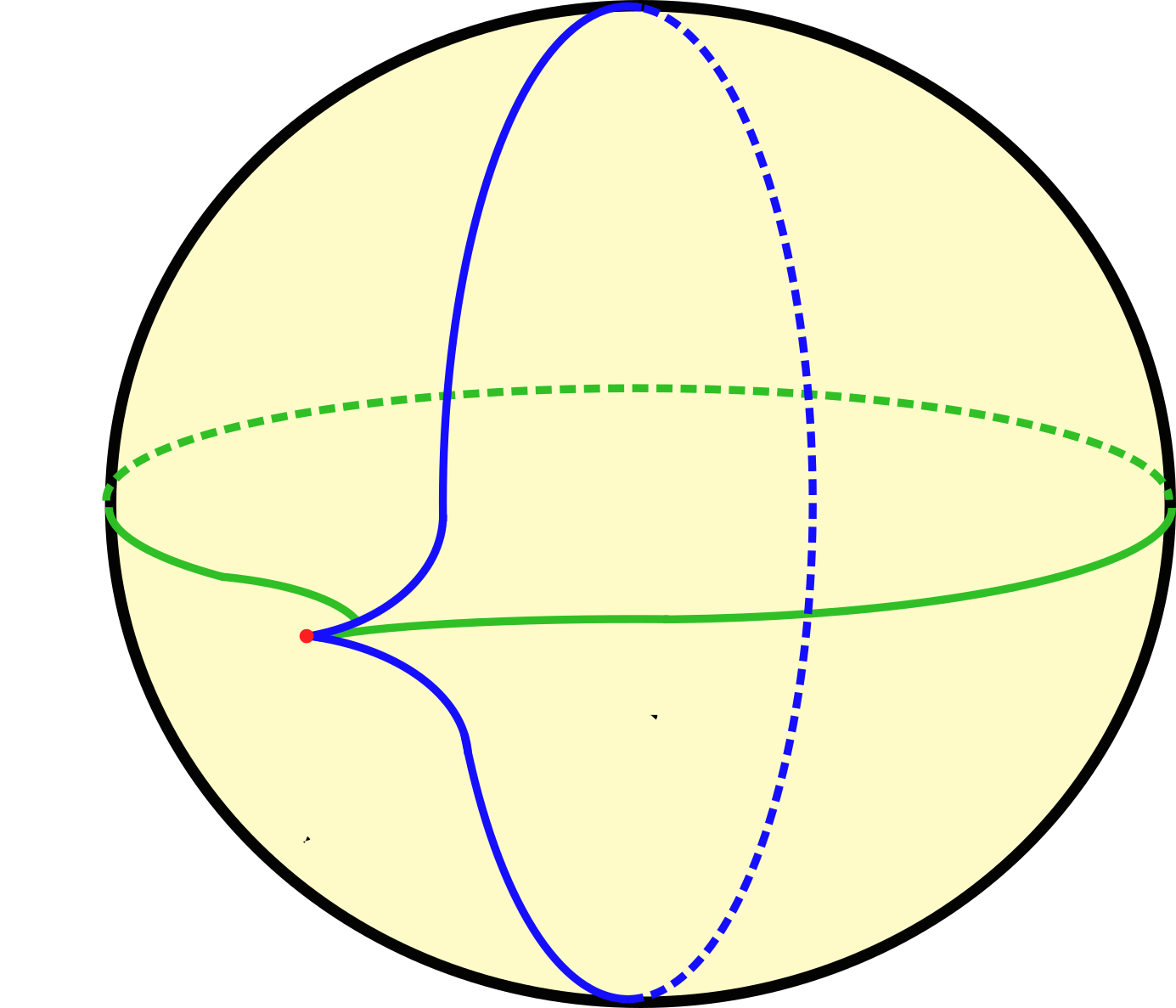}
\caption{The compact moduli of Hecke parameters for $d=5$: $V_0$ yellow, $V_1$ blue, $V_2$ red, $N_2^0$ green. \label{fig:Heckemoduli}}
\end{center}
\end{figure}

\begin{exam} Let us take a closer look at the compact moduli of Hecke parameters for $d=5$. This example is sketched in Figure \ref{fig:Heckemoduli}. 

In this case, we have three strata of Hecke parameters by Lemma \ref{lemm:u-coord:2}: $\Heck_5=V_0 \cup V_1 \cup V_2$, where $V_0\cong \C^2, V_1 \cong \C$, and $V_0 \cong \{0\}$. Proposition \ref{prop:N_l^n} defined three subset $N_l^n$, where $N_1^0= V_0 \cup V_1$ is of dimension $2$, $N_2^1=V_1 \cup V_2$ is of dimension $1$, and $N_2^0$ is a one-dimensional subset of $V_0$ with $V_2$ in its closure. 

Because the glueing maps are algebraic, we can give $\Heck_5/G_D$ the structure of a scheme using a pushout-construction (cf. \cite{Schw05}). This is the universal scheme structure $(\Heck_5/G_D, \O)$, such that for any other scheme structure $(\Heck_5/G_D, \O')$ the inclusions of the subsets $V_i$ and $N_l^n$ into  \linebreak $(\Heck_5/G_D, \O')$ factor through $(\Heck_5/G_D, \O)$. However, because we do not glue along closed subschemes, this scheme structure has some pathologies as explained in \cite{Schw05} Example 3.2. In particular, it is not the scheme structure we will obtain below by embedding the space of Hecke parameters in the singular Hitchin fibre. 
\end{exam}

\subsection{Global fibring over twisted Prym varieties}\label{ssec:glob:odd}
We will show that the singular fibres with locally irreducible spectral curve fibre over twisted Prym varieties with fibres given by the compact moduli of Hecke parameters. As a first step we identify the twisted Prym varieties of all strata.

\begin{defi}\label{def:eigtw} Let $\M_\lambda^\sigma$, such that $\Fix(\sigma)=Z(\lambda)$. Define
\begin{align*} \Eig_{\mathsf{tw}}: \M_\lambda^\sigma= \bigcup\limits_D \mathcal{S}_D \rightarrow \Pr_\Lambda(Y), \quad (E,\Phi) \mapsto \Eig_D(E,\Phi)(-\tfrac12 D).
\end{align*}
\end{defi}

\begin{rema} This is well-defined, because $D$ has only even coefficients. When $\Fix(\sigma) \neq Z(\lambda)$, then there is no canonical way to identify the twisted Prym varieties of different strata. See Section \ref{sec:global_even} for more details.
\end{rema}

We defined two kinds of $u$-coordinates: First in Proposition \ref{prop:u-coord}, when parametrising the strata and second in Lemma \ref{lemm:u-coord:2}, when parametrising $V_n \subset \Heck_d$. They are equivalent in the following way.

\begin{prop}\label{prop:Appl:Hecke} Let $p \in Y$ and $\Lambda=d \cdot p$. Let $0 \leq n \leq \frac{d-3}{2}$ and $(a,b) \in V_n \subset \Heck_d $. Let $L \in \Pr_{\Lambda}(Y)$, choose a frame $s$ of $L$ in neighbourhood of $p$ and let $\alpha=as_++bs_-\in H^0(\Lambda,(E,\hat{\sigma}))$. Then 
\[ (\hat{E}_L^{(\Lambda,\alpha)},\hat{\Phi}_L^{(\Lambda,\alpha)}) \in \S_{2n \, p},
\] its image under $\Eig_{2n\ p}$ is $L(n\ p)$ and the $u$-coordinate defined in Proposition \ref{prop:u-coord} is given by
\[ u_n(\alpha) \in \C^{\frac12 (d-2n-1)} \subset \P^{\frac12 (d-2n-1)},
\] with $u_n$ defined in Lemma \ref{lemm:u-coord:2}. 
\end{prop}
\begin{proof}
The Higgs field of the Hecke transformation at $p$ is given by
\begin{align} \hat{\Phi}_L=\hat{\Phi}_L^{(\Lambda,\alpha)}=\hat{\psi}_{01}^{-1} \Phi_L \hat{\psi}_{01} =\begin{pmatrix} \frac{a}{b}z^d & b^2-a^2 \\ \frac{z^{2d}}{b^2} & -\frac{a}{b}z^d \end{pmatrix} \d z \label{equ:Higgsfield}
\end{align} with respect to the induced frame on $\hat{E}^{(\Lambda,\alpha)}$.
A section of the eigenline bundle $L$ at $p$ is given by 
\[ s=\begin{pmatrix} b+a \\ z^db^{-1}
\end{pmatrix}.
\]  Let $s=z^{\ord_p(a+b)}\tilde{s}$, then $\tilde{s}$ defines a non-vanishing section of the eigenline bundle $\hat{L} = \Ker(\hat{\Phi}_L-\lambda \id_{\hat{E}_L})$. 
In particular, $\hat{L}=L(np)$ and
\[ (\hat{E}_L^{(\Lambda,\alpha)},\hat{\Phi}_L^{(\Lambda,\alpha)}) \in \S_{2n\, p}.
\] 
To compute the $u$-coordinate at $p$, let us first assume that $n$ is even, i.e.\ $\ord_p(a+b)=\ord_p(b)$. Then
\[ \tilde{s}=\begin{pmatrix} \tilde{b} + \tilde{a} \\ z^{d-2n}\tilde{b}^{-1}
\end{pmatrix}
\] with $\tilde{a}$ an odd polynomial of degree $d-n-2$ and $\tilde{b}$ a non-vanishing even polynomial of degree $d-n-1$. The sections $s_{\pm}$ are given by 
\[ s_+=\tilde{s}+\sigma^*\tilde{s}= \begin{pmatrix} \tilde{b} \\ 0 \end{pmatrix}, \qquad s_-=\tilde{s}-\sigma^*\tilde{s}= \begin{pmatrix} \tilde{a} \\ z^{d-n} \tilde{b}^{-1} \end{pmatrix}.
\] Hence, the $u$-coordinate as defined in Proposition \ref{prop:u-coord} is given by 
\[ u=\frac{\tilde{a}}{\tilde{b}} \mod z^{d-2n}.
\] With respect to the basis $z,z^3, \dots, z^{d-2n-2}$, this exactly gives the coordinates $u_n$ defined in Lemma \ref{lemm:u-coord:2}. When $n$ is odd, a similar consideration gives the result.
\end{proof}

Let $D$ be a $\sigma$-Higgs divisor on $(Y,\sigma,\lambda)$. Define
\[ \Heck_D:= \bigtimes\limits_{p \in \supp(D)} \Heck_{D_p}.
\] 

\begin{prop}\label{prop:topo:fibreing} 
Consider $\M_\lambda^\sigma$, such that $Z(\lambda)=\Fix(\sigma)$. Then the map $\Eig_{\mathrm{tw}}: \M_\lambda^\sigma \rightarrow \Pr_\Lambda(Y)$ is a topological fibre bundle with fibres given by the compact moduli of Hecke parameters $\Heck_D/G_D$. 
\end{prop}
\begin{proof}
By definition, it is clear that $\Eig_{\mathrm{tw}}$ is continuous on each stratum. From Proposition \ref{prop:Appl:Hecke}, it is continuous under the degeneration from one stratum to another. Let $U$ a union of neighbourhoods of $Z(\lambda)$ and $V \subset \Pr_\Lambda(Y)$ open such that there exists a local frame of the universal bundle
\[ s: U \times V \rightarrow \mathcal{L}
\] (cf. the proof of Proposition \ref{prop:Hecke_bundle}). By applying Hecke transformations we obtain a commuting diagram
\begin{center}
\begin{tikzpicture}\matrix (m) [matrix of math nodes,row sep=3em,column sep=4em,minimum width=2em]
  {
     \Heck_D/G_D \times V& \Eig_{\mathrm{tw}}^{-1}(V)  \\
    V  & \\
  };
  \path[-stealth]
    (m-1-1) edge node [above] {$\Heck$} (m-1-2)
          	edge node [left] {$\mathrm{pr}_2$}(m-2-1)
    (m-1-2) edge node [below] {\qquad $\Eig_{\mathrm{tw}}$} (m-2-1);
\end{tikzpicture}
\end{center}
The identification of $u$-coordinates in the previous proposition shows that this map is bijective. 
\end{proof}

Following paragraph 2.49 in \cite{KaKa83}, an analytic subset of a complex space is called reducible, if it is the union of proper analytic subsets. Let $X$ a complex space and $\mathrm{Sing}(X)$ the singular set. Then a irreducible component $Z \subset X$ is defined as the closure of a connected component $X\setminus \mathrm{Sing}(X)$. An irreducible component is an irreducible analytic subset.

\begin{coro}\label{coro:irred:odd} The space $\M_\lambda^\sigma$ with $Z(\lambda)=\Fix(\sigma)$ is an irreducible complex space. In particular, it is connected.
\end{coro}
\begin{proof} The space of Hecke parameters of the highest stratum is connected. The twisted Prym variety $\Pr_\Lambda(Y)$ is connected as long as there exists a branch point of $p$ (see \cite{HaPa12}). Furthermore the closure of the highest stratum is the whole singular Hitchin fibre by Theorem \ref{theo:moduli_Hecke} and the previous proposition. In particular, the set of non-singular points is connected and hence $\M_\lambda^\sigma$ is irreducible.
\end{proof}

\begin{rema}
We want to point out that connectedness was already shown in \cite{GO13}. Though not stated there, the irreducibility also follows from Proposition 4.5 and Theorem 6.1 of \cite{GO13} together with the connectedness of the Prym variety. 
\end{rema}

\begin{theo}\label{theo:anal:fibreing} The map $\Eig_{\mathrm{tw}}: \M_\lambda^\sigma \rightarrow \Pr_\Lambda(Y)$ is holomorphic. In particular, the compact moduli of Hecke parameters $\Heck_\Lambda/G_\Lambda$ is a complex space.
\end{theo}
\begin{proof}
We will use a version of the Riemann extension theorem for complex spaces to prove the theorem. To do so, we have to reduce the problem to codimension 2. Let us again assume that there is only one higher order zero of $\lambda$. We saw in Proposition \ref{prop:N_l^n} that an open neighbourhood $N_1^0/G_\Lambda$ of the first stratum $V_1$ in the zeroth stratum $V_0$ can be identified with an open non-singular toric subspace of a weighted projective space $\P(1,1,2,\dots,n)$ . Gluing this open subset to $V_0$, we obtain a complex manifold $V=V_0 \cup V_1$ of Hecke parameters of the zeroth and first stratum. We can build a holomorphic fibre bundle $F_{01}$
\[ V \rightarrow F_{01} \rightarrow \Pr_{\Lambda}(Y)
\] by choosing local frames of $L \in \Pr_{\Lambda}(Y)$ around $Z(\lambda)$. Through Hecke transformations, we obtain an analytic map to $\S_0 \cup \S_1$, such that the following diagram commutes
\begin{center}
\begin{tikzpicture}\matrix (m) [matrix of math nodes,row sep=3em,column sep=4em,minimum width=2em]
  {
     F_{01} & \S_0 \cup \S_1  \\
    \Pr_{\Lambda}(Y)  & \\
  };
  \path[-stealth]
    (m-1-1) edge node [above] {$\Heck$} (m-1-2)
          	edge (m-2-1)
    (m-1-2) edge node [below] {\qquad $\Eig_{\mathrm{tw}}$} (m-2-1);
\end{tikzpicture}
\end{center}
Hence $\Eig_{\mathrm{tw}}$ is holomorphic on $\S_0 \cup \S_1$. 

To extend it, we use the Riemann extension theorem (Thm. I.12.13 in \cite{GPR94}) for reduced locally pure dimensional complex spaces. By Theorem \ref{coro:irred:odd}, $\M_\lambda^\sigma$ is an irreducible complex space. Furthermore, the Hitchin map is flat and therefore its fibres are locally pure dimensional (see Thm. II.2.13 in \cite{GPR94}). Let $p \in \M_\lambda^\sigma\setminus (\S_0 \cup \S_1)$. For a small neighbourhood $U \subset \M_\lambda^\sigma$ of $p$, we can choose coordinate functions 
\[ f: V \subset \Pr_\Lambda(Y) \rightarrow \C^{\dim \Pr_\Lambda},
\] such that $\Eig_{\mathrm{tw}}(U \cap (\S_0 \cup S_1)) \subset V$. Then $f \circ \Eig_{\mathrm{tw}}$ define holomorphic functions on $U$ away from a analytic subset of codimension 2. By the extensions theorem they extend to $U$ meromorphically. We already showed that $\Eig_{tw}$ as defined in Definition \ref{def:eigtw} is a continuous extension. Hence $\Eig_{\mathrm{tw}}$ is holomorphic.
\end{proof}

In conclusion, we obtain the following description of singular $\SL(2,\C)$-Hitchin fibres with locally irreducible spectral curve.

\begin{theo}\label{theo:global_fibreing} Let $q_2 \in H^0(X,K^2)$ with only zeros of odd order. Then $\H^{-1}(q_2)$ is holomorphic fibre bundle 
\[ \Heck_\Lambda / G_\Lambda \rightarrow \H^{-1}(q_2) \rightarrow \Pr_\Lambda(\tilde{\Sigma}).
\] In particular, the singular Hitchin fibre is an irreducible complex space.
\end{theo}

\subsection{The first degenerations}\label{ssec:exam}\ \\
\textbf{Zeros of order 3}
Let $q_2 \in H^0(X,K^2)$ with one zero of order 3, such that all other zeros are simple. In this case, $G_\Lambda \cong \C^*$ is reductive and it is easy to see that the compact moduli of Hecke parameters is given by 
\[  \Heck_3/ \C^* \cong \P^1.
\] So as a direct consequence of Theorem \ref{theo:anal:fibreing}, we obtain:
\begin{coro}\label{coro:zero_order3} Let $q_2 \in H^0(X,K^2)$ with $k$ zeros of order $3$, such that all other zeros are simple. Then the singular Hitchin fibre is a holomorphic fibre bundle
\[ (\P^1)^k \rightarrow \H^{-1}(q_2)  \xrightarrow{\Eig_{\mathsf{tw}}} \Pr_{\Lambda}(\tilde{\Sigma}).
\] In particular, $\H^{-1}(q_2)$ is a toric complex space.
\end{coro}

\noindent \textbf{Zeros of order 5}
Let us now consider $q_2 \in H^0(X,K^2)$ with zeros of order 5.
\begin{prop}\label{prop:order5} The compact moduli of Hecke parameters $\Heck_5/G_D$ is a toric complex space normalised by $\P(1,1,2)$.
\end{prop}
\begin{proof} In Proposition \ref{prop:N_l^n}, we defined an isomorphism from the 
\[ N_1^0 \rightarrow \P(1,1,2)\setminus\{(y_0:0:y_2)\}.
\] Its inverse is given by
\[ \P(1,1,2) \setminus \{( y_0:0:y_2)\} \rightarrow \Heck_d/G_\Lambda, \quad (y_0:y_1:y_2) \mapsto \begin{pmatrix} y_1^2z \\ y_0y_1+y_2z^2 
\end{pmatrix}. 
\] This map naturally extends to $(0:0:1)$ by mapping it onto $V_2$ consisting of a single point. If $y_0 \neq 0$, the image lies in $V_0$ and 
\[ (u_0 \circ \psi)(y_0:y_1:y_2)= \frac{y_1}{y_0}z+\frac{y_2}{y_0^2}z^3.
\] Therefore, it extends holomorphically to $y_0 \neq0, y_1=0$. The map is biholomorphic away from the point in the lowest stratum, which is a fixed point of the full-dimensional torus action on $\P(1,1,2)$. Hence, we can pushforward the torus action to the moduli of Hecke parameters.
\end{proof}

\begin{coro}\label{coro:zero_order5} Let $q_2 \in H^0(X,K^2)$ with $k$ zeros of order $3$ and $l$ zeros of order $5$, such that all other zeros are simple. Then, up the normalisation, $\H^{-1}(q_2)$ is a holomorphic fibre bundle 
\[  (\P^1)^k \times (\mathbb{P}(1,1,2))^l \rightarrow \H^{-1}(q_2) \xrightarrow{\Eig_{\mathsf{tw}}} \Pr_{\Lambda}(\tilde{\Sigma}).
\] In particular, $\H^{-1}(q_2)$ is a toric complex space.
\end{coro}

\section{Singular Fibres with Irreducible Spectral Curve}\label{sec:global_even}
Whenever there exist locally reducible singularities of the spectral curve, the singular Hitchin fibres do not fibre over twisted Prym varieties. However, we can still describe the degeneration to lower strata using Hecke transformations. In Section \ref{sec:strat}, it was more convenient to parametrize the extra data at even zeros with extension data. We will reinterpret these extra data as Hecke parameters now. \\

Fix $\M_\lambda^\sigma$, such that $\{y, \sigma(y)\}=Z(\lambda)\setminus\Fix(\sigma)$ and all other zeros of $\lambda$ are simple. Let $D$ be an associated $\sigma$-Higgs divisor. Let $L \in \Pr_{\Lambda-D}$ and $(E,\Phi) \in \M_\lambda^\sigma$ obtained from $(E_L,\Phi_L)$ by applying the unique Hecke transformation at all simple zeros. Choose frames $s_1 \in H^0(U,L), \ s_2 \in H^0(U,\sigma^*L)$ for a neighbourhood $U$ of $y$ and let 
\[ s_+:=s_1 \oplus s_2, \quad s_-:=s_1 \oplus - s_2. 
\] the induced frame of $E\rest_U=(L \oplus \sigma^*L) \rest_U$. 
\begin{prop}\label{prop:Hecke:even} Let $l=(\Lambda-D)_y$ and $\alpha=as_++bs_- \in H^0(l y,E)^*$, such that $a(0) \neq \pm b(0)$. Then 
\[ \left( \hat{E}^{(y + \sigma y, \alpha + \sigma^*\alpha)}, \hat{\Phi}^{(y + \sigma y, \alpha + \sigma^*\alpha)} \right) \in S_D \subset \M_\lambda^\sigma  
\] and the extension datum at $y$ introduced in Proposition \ref{prop:even:zeros} is given by 
\[ \left[ \frac{b+a}{b-a} \ s_1^2 \d z \right] \in H^0(ly,L^2K).
\]
\end{prop}
\begin{proof}
This is a local computation from the description of the Higgs field given in (\ref{equ:Higgsfield}). 
\end{proof}

From this description, we see that for the Hecke parameters at even zeros of the quadratic differential, there are two different ways to degenerate to lower strata: 
\begin{itemize}
\item[i)] By degenerating to lower strata in the moduli spaces of Hecke parameters, i.e. allowing $\alpha$ to vanish. Here the eigenline bundle is twisted by a $\sigma$-invariant divisor $D+\sigma^*D$. 
\item[ii)] By imposing 
\[ a \equiv b \mod z^l \quad \text{or} \quad a \equiv -b \mod z^l
\] for some $l \leq  \ord_{y}(\lambda)$, while $a(0)$, $b(0) \neq 0$. In this case, the eigenline bundle is twisted by divisors $l y$ or $l \sigma(y)$, respectively. 
\end{itemize}
Consonant with the previous section, we can find a compactification of the Hecke parameters of the highest stratum by allowing Hecke parameters in $\alpha \in H^0(\Lambda_y y, E)$. Define 
\[ \Heck_{\Lambda_{y}}:=H^0(\Lambda_{y}y,E)/\sim,
\] where $\sim$ denotes the analogue of relation ii) of Lemma \ref{lemm:equ:ii}. Along the lines of Section \ref{ssec:mod:Heck}, we can study the quotient of $\Heck_{\Lambda_{y}}$ by the non-reductive group action of $H^0(\Lambda_{y} \, y, \O_Y^*)$ and obtain a topological model by gluing toric subspaces of weighted projective spaces. Following Section \ref{ssec:glob:odd} one proves:

\begin{theo}\label{theo:hecke_to_fibre} Let $q_2 \in H^0(X,K^2)$ with one zero $x \in X$ of order $2d$, such that all other zeros are simple. Let $\tilde{\pi}^{-1}(x)=\{y,\sigma y\}$ and $L \in \Pr_\Lambda(\tilde{\Sigma})$. Denote by $(E,\Phi) \in \tpst \H^{-1}(q_2)$ the unique $\sigma$-invariant Higgs bundle obtained by applying Hecke transformations to $(E_L,\Phi_L)$ at the simple zeros of $q_2$. There is a continuous injective map
\[ T_L: \Heck_{d}/H^0(d y, \O_Y^*) \rightarrow \tpst \H^{-1}(q_2),
\] defined by applying Hecke transformations to $(E,\Phi)$ at $x \in X$. Its image is the closure of $\tpst \Eig_0^{-1}(L)$ in $\tpst \H^{-1}(q_2)$ and is given by
\[ \bigcup\limits_{l_1+l_2 \leq d} \tpst \Eig^{-1}_{D(l_1,l_2)}(L(l_1y+l_2\sigma y))
\] with $D(l_1,l_2)=(l_1+l_2)y+(l_1+l_2)\sigma y \in \Div^+(Y)$.
\end{theo}
\begin{proof}
This theorem is proven by adapting Proposition \ref{prop:Appl:Hecke} to Hecke transformations at the even zeros of $q_2$. This allows to compute the eigenline bundles of the limit points determining the image of $T_L$.
\end{proof}

Define $F_{q_2}$ as the topological fibre bundle over $\Pr_\Lambda$ with fibres given by the moduli of Hecke parameters (cf. Proposition \ref{prop:Hecke_bundle}). We can define a continuous map $T: F_{q_2} \rightarrow \H^{-1}(q_2)$ by applying $T_L$ fibrewise. However, as we will see below, this map is no longer injective. It has the property that it makes the following diagram commute
\begin{center}
\begin{tikzpicture}\matrix (m) [matrix of math nodes,row sep=3em,column sep=4em,minimum width=2em]
  {
    T^{-1}(\S_0) & \S_0 \\
    \Pr_\Lambda  & \\
  };
  \path[-stealth]
    (m-1-1) edge node [above] {$T$}(m-1-2)
          	edge (m-2-1)
    (m-1-2) edge node [right] {\quad$\Eig_0$}(m-2-1);
\end{tikzpicture}
\end{center}
But there is no way to extend the fibreing to the whole singular fibre. This was already encountered in \cite{GO13} and \cite{Hi19}. To illustrate why these two properties fail, we describe the case of zeros of order 2. \\

\begin{exam}[Zeros of order 2]\label{exam:m=1} Let $q_2 \in H^0(X,K^2)$ having a zero $x \in X$ of order two, such that all other zeros are simple. 
Let $\{y, \sigma y\}= \tilde{\pi}^{-1}(x)$. The compact moduli of Hecke parameters at $x$ is given by
\[ \left(H^0(y,E_L)\setminus\{0\}\right)/H^0(y,\O_Y^*) = \left(\C^2\setminus \{0 \}\right) /\C^* = \P^1.
\] In this case, the stratification of Hecke parameters by vanishing order is trivial. However, from Theorem \ref{theo:strat:sl2C} the stratification of $\H^{-1}(q_2)$ has two strata. One, where the Higgs field is non-vanishing and one, where it is diagonalizable and vanishes at $x$ of order 1. 

Let $L \in \Pr_\Lambda(Y)$ and $\alpha \in F_{q_2}\rest_L$. Let $U \subset X$ a neighbourhood of $x$. Choosing frames $s_1 \in H^0(U,L)$ and $s_2 \in H^0(U,\sigma^*L)$, the Hecke parameter $\alpha$ can be written as $\alpha=as_++bs_-$. Then the Higgs field of $T(L,\alpha)$ is given by formula (\ref{equ:Higgsfield}) in Section \ref{ssec:glob:odd}. Hence,
\[ T(L,\alpha) \in S_0 \quad \Leftrightarrow \quad a_0 \neq \pm b_0.
\] Furthermore, it is easy to check that for $a_0=b_0$ the eigenline bundle of the Hecke transformation is given by $L(y)$, whereas for $a_0=-b_0$ it is given by $L(\sigma y)$. We conclude that for given $L \in \Pr_{\Lambda}(\tilde{\Sigma})$ 
\[ T\left(L,s_++s_-\right) = T\left(L(y-\sigma y )),s_+-s_-\right).
\] In particular, $T$ is not injective and the fibreing can not be extended. 
However, Theorem \ref{theo:hecke_to_fibre} defines a holomorphic map from a holomorphic $\P^1$-bundle over $\Pr_\Lambda$ to the Hitchin fibre, which has a holomorphic inverse on the dense stratum $\S_0$. Hence, the normalisation of the Hitchin fibre $\H^{-1}(q_2)$ is a $\P^1$-bundle over a Prym variety. 
\end{exam}

\begin{rema} In upcoming work of Xuesen Na and the author \cite{HoNa}, we will show that the compact moduli of Hecke parameters at locally reducible singularities of the spectral curve are complex analytic spaces. Then the map $T$ defines a one-sheeted analytic covering in the language of \cite{GPR94}. The analogue of a birational morphism in the analytic category. 
\end{rema}

\begin{exam}[zeros of order 4] \label{exam:m=2} 
Let $q_2\in H^0(X,K^2)$, such that there is one zero $x \in X$ of order $4$ and all other zeros are simple. Then up to normalisation the singular Hitchin fibre is given by a holomorphic $\P(1,1,2)$-bundle over $\Pr_\Lambda(\tilde{\Sigma})$. \\
The proof is similar to the proof of Proposition \ref{prop:order5}. Choose a local coordinate $(U,z)$ centred at $y \in \tilde{\pi}^{-1}(x)$ and let
\[ \Heck_2=\left\{ \alpha=\begin{pmatrix} a_0 +a_1 z \\ b_0 +b_1 z \end{pmatrix} \in H^0(2 \cdot y,\O_Y^2) \right\} / \sim,
\] where $\alpha \sim \alpha' \Leftrightarrow a_0=b_0=0$. The compact moduli of Hecke parameters at $x \in X$ is the quotient of $\Heck_2$ by $G_2:=H^0(2 y, \O_Y^*)$. We defined holomorphic $G_2$-invariant functions
\[ \Heck_2\setminus \{[0]\} \rightarrow \P(1,1,2)\setminus\{(0:0:1)\}, \quad  \alpha \mapsto (b_0:a_0:b_1a_0-b_0a_1).
\] An inverse is given by 
\begin{align*} \Psi: \P(1,1,2)\setminus\{(0:0:1)\} &\rightarrow (\Heck_2\setminus \{[0]\})/G_2, \\ (y_0:y_1:y_2) \quad &\mapsto \begin{bmatrix} y_1^2 \\ y_0y_1 + y_2 z \end{bmatrix}.
\end{align*} $\Psi$ extends to $(0:0:1)$ by mapping it to $[0] \in \Heck_2/G_2$. The holomorphic transition functions of $\S_0$ define a holomorphic fibre bundle
\[ \P(1,1,2) \rightarrow F \rightarrow 	\Pr_\Lambda(\tilde{\Sigma})
\]
with a holomorphic map $T \circ \Psi: F \rightarrow \H^{-1}(q_2)$. The $\P(1,1,2)$-bundle over $\Pr_\Lambda(\tilde{\Sigma})$ is normal and $T \circ \Psi$ is biholomorphic on the highest stratum. Hence, this is the normalisation of $\H^{-1}(q_2)$.
\end{exam}

\begin{coro}\label{coro:even_irred} Let $q_2 \in H^0(X,K^2)$ with at least one zero of odd order. Then $\H^{-1}(q_2)$ is irreducible.
\end{coro}
\begin{proof}
Theorems \ref{prop:topo:fibreing} and \ref{theo:hecke_to_fibre} show that $\overline{\S_0}= \H^{-1}(q_2)$. Furthermore, $\S_0$ is a $(\C^*)^{r_1} \times \C^{r_2}$-bundle over a twisted Prym variety, which is connected as long as there exists a branch point of $\tilde{\pi}: \tilde{\Sigma} \rightarrow X$ (see \cite{HaPa12}). In particular, the smooth points of $\H^{-1}(q_2)$ are connected and hence $\H^{-1}(q_2)$ is irreducible. 
\end{proof}

We already encountered above, that the case of quadratic differentials with only zeros of even order is very special. As we saw in Proposition \ref{prop_pi*_inj} the pullback of Higgs bundles to the normalized spectral curve is not injective. A second issue is that $$\Prym(\tilde{\Sigma})=\Nm^{-1}(\O_{X})$$ is not connected. 
A detailed discussion of this issue can be found in \cite{Mu71} (and \cite{HaPa12}). We give a short recall of Mumford's approach here. 
\begin{lemm}
Let $L \in \Prym(\tilde{\Sigma})$, then there exists a line bundle $M$ on $\tilde{\Sigma}$ of degree 1 or 0, such that $L\cong M \otimes \sigma^*M^{-1}$. 
\end{lemm}

The two connected components can be seen in two different ways: First they can be defined as
\[ P_i=\{ M\otimes \sigma^*M^{-1} \mid M \in \Pic^i(\tilde{\Sigma})\}
\] for $i=0,1$. 
To see that these define different components Mumford showed that the function
\[ \Prym(\tilde{\Sigma}) \rightarrow \Z_2, \quad L \mapsto \dim H^0(L\otimes \tpst K^{\frac12}) \mod 2
\] is locally constant and is equal to $i \mod 2$ restricted to $P_i$. \\
Secondly, one can see the different components in the following way:
Let $L=M\otimes \sigma^*M^{-1}$ and $\O(D)=M$. Let $D=D_+-D_-$, where $D_\pm$ are effective divisors then $\deg(D_+)=\deg(D_-)+i$. We can also decompose $D-\sigma^*D=(D-\sigma^*D)_++(D-\sigma^*D)_-$ into effective divisors. Then $\deg((D-\sigma^*D)_+)=2\deg(D)-i \equiv i \mod 2$. This does only depend on the parity of $\deg(M)$ again. So if we represent a element $L \in \Pr_0(\tilde{\Sigma})$ with a divisor $C$ with $C + \sigma^*C=0$, then $\deg(C_+) \mod 2$ tells us in which connected component it lies.

\begin{theo}\label{theo:red:even:zero} Let $q_2 \in H^0(X,K^2)$, such that all zeros of $q_2$ have even order. Then the singular Hitchin fibres $\H^{-1}(q_2)$ is connected and has four irreducible components.
\end{theo}
\begin{proof} We showed in Proposition \ref{prop_pi*_inj}, that $\H^{-1}(q_2)$ is a branched two-to-one covering of $\M^\sigma(\tilde{\Sigma},\pi^*K,\lambda)$. Example \ref{exam:branchpoint} shows that one can always find a branch point in the lowest stratum. In particular, we can conclude connectedness of the singular fibre from the connectedness of $\M^\sigma(\tilde{\Sigma},\pi^*K,\lambda)$. 

Let $L \in \Nm^{-1}(I(-\frac12\div(q_2)))$ and $p \in \tilde{\pi}^{-1}Z(q_2)$. As we saw in Theorem \ref{theo:hecke_to_fibre}, we can degenerate from $\Eig_0^{-1}(L) \subset \S_0$ and $\Eig_0^{-1}(L(p - \sigma(p))) \subset \S_0$ to one and the same point in the lower stratum. Moreover we just saw, that we interchange the connected component of the twisted Prym variety by tensoring with $\O(p - \sigma(p))$ for $p \in \tilde{\Sigma}$. Hence, $\H^{-1}(q_2)$ is connected.

However, the fibre is reducible, because $\S_0$ is disconnected. It decomposes by restricting the $\C^*\times \C^n$-bundle to the two connected components of $\Nm^{-1}(I(-\frac12 \div(q_2))$. The closures of these two connected components of $\S_0$ define two irreducible components of $\M^\sigma(\tilde{\Sigma},\pi^*K,\lambda)$. As the pullback of a branch point of $\tpst$ has an extra automorphism it is strictly polystable and hence lies in the lowest stratum. In particular, the pullback is two-to-one on the highest stratum. Thereby, $\H^{-1}(q_2)$ has four irreducible components. 
\end{proof}

\section{Real Points in Singular Hitchin Fibres}\label{sec:real}
In this section, we are going to study $K$-twisted $\SL(2,\R)$-Higgs bundles with irreducible and reduced spectral curve. We will show that for each stratum they are parametrised by the two-torsion points of the Prym variety and a discrete choice of Hecke parameters at the even zeros of the quadratic differential. The result for regular Hitchin fibres was considered in \cite{Sc13b,PNThesis}.\\

A line bundle $L \in \Prym(\tilde{\Sigma})$ is a two-torsion point, if $L^2\cong \O_X$. Under the Prym condition, this is equivalent to $\sigma^*L \cong L$. We call $L \in \Pr_N(\tilde{\Sigma})$ $\sigma$-symmetric, if $\sigma^*L \cong L$. Choosing a $\sigma$-symmetric base point for the simply transitive action of $\Prym(\tilde{\Sigma})$ on $\Pr_N(\tilde{\Sigma})$ the two-torsion points are bijectively mapped on the $\sigma$-symmetric points. Recall that the definition of $\sigma$-invariant holomorphic line bundle required the lift $\hat{\sigma}$ to restrict to the identity at all ramification points (cf. Definition \ref{def:sigma-inv}).

\begin{theo}\label{theo:realpts:odd} Let $q_2 \in H^0(X,K^2)$ a quadratic differential, such that all zeros have odd order and $D \in \Div(X)$ an associated Higgs divisor. Then the $\SL(2,\R)$-Higgs bundles in $\S_D\subset \H^{-1}(q_2)$ are parametrised by the $\sigma$-symmetric points of $\Pr_{\Lambda-\tpst D}(\tilde{\Sigma})$. 
\end{theo}
\begin{proof}
Let $N:=\Lambda-\tpst D$ and $L \in \Pr_{N}(\tilde{\Sigma})$, such that there exists an isomorphism $\phi: \sigma^*L \rightarrow L$. Then $\phi$ is unique up to $\pm \id$ and restricts to $\pm 1$ at each $p \in \Fix(\sigma)=\tilde{\pi}^{-1}Z(q_2)$. Choose a frame $s \in H^0(U,L)$ at $p\in \Fix(\sigma)$, such that $s =  \pm \phi( \sigma^*s)$ for $\phi_p=\pm 1$ respectively. Such a frame is uniquely defined up to multiplying by a $\sigma$-invariant holomorphic function and therefore defines a unique $u$-coordinate (cf. Proposition \ref{prop:u-coord}). The induced frame $s_+,s_-$ defines a global splitting 
\[ (E_L,\Phi_L)= \left(L \oplus L, \begin{pmatrix} 0 & \lambda \\ \lambda & 0 \end{pmatrix} \right).
\] We decompose $N=N_++N_-$, such that 
\[ \supp N_\pm = \{ p \in \Fix(\sigma) \mid \phi \rest_p = \pm \id \}.
\] The Hecke transformation of $(E_L,\Phi_L)$ at $N$ in direction $u=0$ is given by
\[ (\hat{E}_L, \hat{\Phi}_L ) = \left( L(N_-) \oplus L(N_+), \begin{pmatrix} 0 & \frac{\lambda \eta_-}{\eta_+} \\ \frac{\lambda \eta_+}{\eta_-} & 0 \end{pmatrix} \right)
\] with $\eta_\pm \in H^0(\tilde{\Sigma}, \O(N_\pm))$ canonical. The induced lift of $\sigma$ to $L(N_\pm)$ is the identity at all ramification points. Hence, $(\hat{E}_L, \hat{\Phi}_L )$ descends to a $\SL(2,\R)$-Higgs bundle on $X$. If we choose $-\phi$ in the beginning the role of $N_\pm$ are interchanged and we obtain a $\SL(2,\R)$-Higgs bundle isomorphic over $\SL(2,\C)$. 

For the converse, consider a $\SL(2,\R)$-Higgs bundle
\[ (E,\Phi)=\left(L \oplus L^{-1}, \begin{pmatrix} 0 & \alpha \\ \beta & 0 \end{pmatrix} \right) \in \S_D.
\] There are divisors $N_\pm \in \Div^+(\tilde{\Sigma})$, such that 
\[ \tpst \alpha= \frac{\lambda \eta_-}{\eta_+}, \quad \tpst \beta  = \frac{\lambda \eta_+}{\eta_-}
\] for $\eta_\pm \in H^0(\tilde{\Sigma},\O(N_\pm))$ canonical. The eigenline bundles are defined by 
\begin{align} (\tpst L)(-N_-) \xrightarrow{(\eta_-,\pm\eta_+)} \tpst L \oplus \tpst L^{-1}. \label{equ:eig}
\end{align} and correspond to a $\sigma$-symmetric point of $\Pr_N(\tilde{\Sigma})$. Furthermore, the induced isomorphism 
\[ \phi: \sigma^*((\tpst L)(-N_-)) \rightarrow (\tpst L)(-N_-)
\] is $-1$ at $\supp N_-$. So we recover $(E,\Phi)$ with the construction in the first part of the proof.	
\end{proof}

\begin{exam}
The pullback $\tpst K^{-\frac12} \in \Pr_\Lambda$ is $\sigma$-symmetric. The corresponding $\SL(2,\R)$-Higgs bundle is the image of the Hitchin section
\[ \left( K^{-\frac12} \oplus K^{\frac12}, \begin{pmatrix} 0 & 1 \\  q_2 & 0
\end{pmatrix}\right) \in \S_0.
\] More generally, if $\deg( D) \equiv 0 \mod 2$, there exist line bundles $M$ on $X$ such that $M^2\cong \O_X(D)$. Then $\tpst (K^{-\frac12}M) \in \Pr_{\Lambda-\tpst D}$ is $\sigma$-symmetric. The corresponding $\SL(2,\R)$-Higgs bundle is of the form
\[ \left(K^{-\frac12}M \oplus K^{\frac12}M^{-1}, \begin{pmatrix} 0 & \eta \\ \frac{q_2}{\eta} & 0\end{pmatrix}\right)
\] with $\eta \in H^0(X,\O(D))$ canonical. These are the only $\SL(2,\R)$-Higgs coming from $\sigma$-invariant line bundles. \\
\end{exam}

\begin{coro} 
Let $q_2 \in H^0(X,K^2)$ be a quadratic differential, such that all zeros have odd order. Then $\H^{-1}(q_2)$ contains
\[ 2^{2g-2}\prod\limits_{p \in Z(q_2)} (\ord_p(q_2)+1)
\] $\SL(2,\R)$-Higgs bundles.
\end{coro}
\begin{proof} By the previous theorem, every stratum contains $2^{2g-2-n}$ $\SL(2,\R)$-points, where $n$ is the number of zeros. At a zero $p \in Z(q_2)$, we have $\frac{\ord_p(q_2)+1}{2}$ possible values for $D$ and hence there are
\[ \prod\limits_{p \in Z(q_2)} \tfrac12 (\ord_p(q_2)+1)
\] different strata.
\end{proof}

For quadratic differentials with zeros of even order, there are two Hecke parameters in each stratum leading to $\SL(2,\R)$-Higgs bundles. Here, we use the description of the extra data at even zeros by Hecke parameters given in Proposition \ref{prop:Hecke:even}.
\begin{theo}\label{theo:realpts:even}
Let $q_2 \in H^0(X,K^2)$ with at least one zero of odd order and $D\in \Div^+(X)$ an associated Higgs divisor. The $\SL(2,\R)$-Higgs bundles in the stratum $\mathcal{S}_D \subset \H^{-1}(q_2)$ are parametrised by the $\sigma$-symmetric points of $\Pr_{\Lambda-\tpst D}(\tilde{\Sigma})$ together with a choice of one of two possible Hecke parameters at every even zero $p$ of $q_2$, where $\tfrac12 \ord_p(q_2) \neq D_p$. In particular, each stratum $\S_D$ contains
\[ 2^{2g-2+n-n_0}
\] $\SL(2,\R)$-Higgs bundles, where $n=\#Z(q_2)$ and 
\[ n_0= \# \{ p \in Z(q_2) \mid \tfrac12 \ord_p(q_2)= D_p\}.
\]
\end{theo}
\begin{proof}
Let $L \in \Pr_{\Lambda-\tpst D}(\tilde{\Sigma})$, such that there exists an isomorphism $\phi: \sigma^*L \rightarrow L$. Fix a choice of $\pm$ at every even zero $p$, such that $\tfrac12 \ord_p(q_2) \neq \tpst D_p$. Let $Z^e \subset Z(q_2)$ be the set of zeros of even order and
\[ N^{\even}= (\Lambda- \tpst D) \rest_{\tilde{\pi}^{-1}Z^e}.
\] Let $N^{\even}=N^{\even}_+ + N^{\even}_-$, such that $N^{\even}_\pm$ is supported at the even zeros assigned a $\pm$ respectively. As seen above, the isomorphism $\phi$ defines a unique Hecke transformation at $\tilde{\pi}^{-1}(p)$ for all $p \in Z(q_2)$ of odd order. Performing these Hecke transformations, we obtain a $\sigma$-invariant Higgs bundle
\[ (E,\Phi)=\left(L_1 \oplus L_2, \begin{pmatrix} 0 & \alpha \\ \beta & 0 \end{pmatrix} \right)
\] on $\tilde{\Sigma}$ with $L_1 \otimes L_2 \cong \O_{\tilde{\Sigma}}(-N^{\even})$, which is locally diagonalizable over all even zeros of $q_2$. If $N^{\even}=0$, this descends to a $\SL(2,\R)$-Higgs bundle on $X$ and we are done. Let $\tilde{p} \in \tilde{\pi}^{-1}Z(q_2)$, such that $N^{\even}_{\tilde{p}} \neq 0$. Choose a frame $s \in H^0(U \cup \sigma(U),L)$ for a neighborhood $U$ of $\tilde{p}$, such that $\phi( \sigma^*s)=s$. Depending on the fixed choice of $\pm$, we define the Hecke parameter $\alpha=[s_\pm]$ with respect to the induced frame $s_+,s_-$. By Proposition \ref{prop:Hecke:even}, this defines a $\sigma$-invariant Higgs bundle 
\[ \left(\hat{E}^{(\tilde{p}+\sigma \tilde{p}, \alpha + \sigma^*\alpha)}, \hat{\Phi}^{(\tilde{p}+\sigma \tilde{p}, \alpha + \sigma^*\alpha)} \right).
\] Performing Hecke transformations like this over all even zeros $p \in X$, such that $\tfrac12 \ord_p(q_2) \neq \tpst D_p$, we obtain 
\[ \left( L_1(N^{\even}_-) \oplus L_2(N_+^{\even}), \begin{pmatrix} 0 & \frac{\alpha \eta_-}{\eta_+} \\ \frac{\beta \eta_+}{\eta_-} &0 \end{pmatrix} \right).
\] This $\sigma$-invariant Higgs bundle descends to a $\SL(2,\R)$-Higgs bundle in the desired stratum. 

For the converse, let 
\[ (E,\Phi)= \left( L \oplus L^{-1}, \begin{pmatrix} 0 & \gamma \\ \delta & 0 \end{pmatrix} \right) \in \S_D.
\] Then 
\[ \tpst (E,\Phi) = \left( \tpst L \oplus \tpst L^{-1}, \begin{pmatrix} 0 & \frac{\lambda \eta_-}{\eta_+} \\ \frac{\lambda \eta_+}{\eta_-} & 0 \end{pmatrix} \right)
\] for divisors $N_\pm \in \Div^+(\tilde{\Sigma})$ with canonical sections $\eta_\pm$. The eigenline bundles are defined by (\ref{equ:eig}) and define a $\sigma$-symmetric element of $\Pr_{\Lambda-\tpst D}$. There is a induced decomposition $N^{\even}=N^{\even}_+ + N^{\even}_-$ and hence a choice of $\pm$ for all $p \in Z^e$, where $N^{\even}_{\tilde{p}} \neq 0$, i.e.\ $\tfrac12 \ord_p(q_2) \neq D_p$. 
\end{proof}

\begin{rema}
The choice of $\pm$ in the previous theorem actually depends on choosing an isomorphism $\phi: \sigma^*L \rightarrow L$. However this isomorphism is unique up to $\pm \id_{\tilde{\Sigma}}$. Choosing $-\phi$ instead of $\phi$ corresponds to switching all $+$ to $-$ and vice versa. For the $\SL(2,\R)$-Higgs bundle, this corresponds to the gauge transformation interchanging the splitting line bundles.
\end{rema}

A general formula for the number of real points in a singular Hitchin fibre would be quite complicated. So let us finish by computing this number in some examples.
\begin{exam}
Let $q_2 \in H^0(X,K^2)$ be a quadratic differential with $d<2g-2$ double zeros and $4g-4-2d$ simple zeros. Then the Hitchin fibre contains
\[ 2^{6g-6-2d} \sum\limits_{k=0}^d \begin{pmatrix} d \\k \end{pmatrix} 2^k 
\] real points.
Let $q_2$ be a quadratic differential with one zero of order $2d<4g-4$ and $4g-4-2d$ simple zeros. Then the number is $(4d-3)2^{6g-6-2d}$.
\end{exam}


\vspace{1.5cm}

\end{document}